\documentclass{amsart}
\usepackage{amsmath,amscd}
\usepackage{amsfonts}
\usepackage{amssymb}
\usepackage[all]{xy}
\usepackage{enumerate}
\usepackage{color}
\input cyracc.def
\font\tencyr=wncyr10
\font\sevencyr=wncyr7
\def\cyr{\tencyr\cyracc}
\def\cyri{\sevencyr\cyracc}
\def\Sh{\textrm{{\cyr Sh}}}
\def\B{\textrm{{\cyr B}}}
\def\Bi{\textrm{{\cyri B}}}
\def\Shi{\textrm{{\cyri Sh}}}
\def\b{\textrm{{\cyr b}}}
\numberwithin{equation}{subsection}
\theoremstyle{plain}
\newtheorem{thm}[subsection]{Theorem}
\newtheorem{prop}[subsection]{Proposition}
\newtheorem{lemma}[subsection]{Lemma}
\newtheorem{cor}[subsection]{Corollary}
\newtheorem*{prin}{The Brauer-Manin obstruction for embedding problems}
\newtheorem*{prin2}{Local-global principle for embedding problems}
\newtheorem*{prin3}{Ad\`elic form of the local-global principle}
\theoremstyle{definition}
\newtheorem{example}[subsection]{Example}
\newtheorem{defn}[subsection]{Definition}
\newtheorem{notn}[subsection]{Notation}
\newtheorem{cont}[subsection]{Contents}
\newtheorem{ackn}[subsection]{Acknowledgement}
\theoremstyle{remark}
\newtheorem{rem}[subsection]{Remark}

\newcommand{\HH}{\mathbb{H}}

\def\alp{{\alpha}}

\def\sig{{\sigma}}

\def\Gam{{\Gamma}}

\DeclareMathOperator{\Ker}{Ker}

\DeclareMathOperator{\Br}{Br}

\DeclareMathOperator{\Hom}{Hom}

\DeclareMathOperator{\Ext}{Ext}

\def\inv{\textup{inv}}

\def\lrar{\longrightarrow}

\def\x{\stackrel}

\begin{document}
\title{The Brauer-Manin obstruction to the local-global principle for the embedding problem}
\author{Ambrus P\'al and Tomer M. Schlank}
\date{April 25, 2017.}
\address{Department of Mathematics, 180 Queen's Gate, Imperial College, London, SW7 2AZ, United Kingdom}
\email{a.pal@imperial.ac.uk}
\address{Einstein Institute of Mathematics, The Hebrew University of Jeru\-sa\-lem, Jeru\-sa\-lem, 91904, Israel}
\email{tomer.schlank@gmail.com}
\begin{abstract} We study an analogue of the Brauer-Manin obstruction to the local-global principle for embedding problems over global fields. We will prove the analogues of several fundamental structural results. In particular we show that the (algebraic) Brauer-Manin obstruction is the only one to weak approximation when the embedding problem has abelian kernel. As a part of our investigations we also give a new, elegant description of the Tate duality pairing and prove a new theorem on the cup product. 
\end{abstract}
\maketitle
\pagestyle{myheadings}
\markboth{Ambrus P\'al and Tomer M. Schlank}{The Brauer-Manin obstruction for the embedding problem}

\section{Introduction}

Let $F$ be an arbitrary field. Fix a separable closure $\overline F$ of $F$ and let $\Gamma=\textrm{Gal}(\overline F|F)$ denote the absolute Galois group of $F$. An embedding problem $\mathbf E$ over $F$ is a diagram:
\begin{equation}\CD @.\Gamma\\
@.@V\psi VV\\
G_1@>\phi>>G_2\endCD\label{1.0.1}
\end{equation}
where $G_1,G_2$ are finite groups, $\phi$ and $\psi$ are group homomorphisms, the map $\phi$ is surjective, and $\psi$ is assumed to be continuous with respect to the Krull topology on $\Gamma$ and the discrete topology on $G_2$. We say that the embedding problem $\mathbf E$ is solvable if there is a continuous homomorphism $\widetilde{\psi}:\Gamma\rightarrow G_1$ which makes the diagram above commutative. We will call such a homomorphism $\widetilde{\psi}$ a solution of $\mathbf E$.  Let $\textrm{Ker}(\mathbf E)=\textrm{Ker}(\phi)$. We will say that two solutions of $\mathbf E$ are conjugate if they are conjugate by an element of Ker$(\mathbf E)$. Conjugacy is clearly an equivalence relation. Let Sol$(\mathbf E)$ denote the set of equivalence classes of this relation. Note that in the subject of field arithmetic it is common to define a solution to an embedding problem to be a surjective continuous homomorphism $\widetilde{\psi}:\Gamma\rightarrow G_1$, and to refer to our notion of solution as a weak solution. In this paper we will not impose the surjectivity condition. 

Assume now that $F$ is a global field. In this case there is an obvious family of obstructions to the solvability of $\mathbf E$ which we will call local obstructions. Let $|F|$ denote the set of all places of $F$ and for every $x\in|F|$ let $F_x$ denote the completion of $F$ with respect to $x$. Fix a separable closure  $\overline F_x$ of $F_x$ and let $\Gamma_x=\textrm{Gal}(\overline F|F)$ denote the absolute Galois group of $F_x$. The choice of an $F$-embedding $\eta_x:\overline F\rightarrow\overline F_x$ induces an injective homomorphism $\iota_x:\Gamma_x\rightarrow\Gamma$ whose conjugacy class is actually independent of these choices. Let $\mathbf E$ be an embedding problem over $F$. Then for every $x\in|F|$ we define the embedding problem $\mathbf E_x$ over $F_x$ associated to $(\mathbf E,x)$ to be the diagram:
$$\CD @.\Gamma_x\\
@.@V\iota_x\circ\psi VV\\
G_1@>\phi>>G_2.\endCD$$
Clearly the embedding problem $\mathbf E_x$ is solvable if the problem $\mathbf E$ is; this is the local obstruction we mentioned above. The analogue of the local-global principle in this setting is the following statement:
\begin{prin2} Let $F$ and $\mathbf E$ be as above and assume that for every $x\in|F|$ there is no local obstruction to the embedding problem $\mathbf E$ at the place $x$. Then $\mathbf E$ is solvable.
\end{prin2}
It is known that the principle above fails by the work of P. Roquette (see \cite{Ro}). The aim of this article is to examine the potential failure of the local-global principle for embedding problems by studying an analogue of the Brauer-Manin obstruction. In particular our results imply that this obstruction explains all failures of the local-global principle in a large class of cases, including the counterexamples provided in \cite{Ro}. In order to formulate these we need to introduce some new concepts and notation. For every non-archimedean $x\in|F|$ let $u_x:\Gamma_x\rightarrow\widehat{\mathbb Z}$ denote the homomorphism onto the Galois group of the maximal unramified extension of $F_x$ in $\overline F_x$. We say that a continuous homomorphism $h:\Gamma_x\rightarrow G$ is unramified if $x$ is non-archimedean and $h$ factors through $u_x$.

Clearly a solution conjugate to an unramified solution is also unramified.
Let $\textrm{Sol}_{un}(\mathbf E_x)$ be the set of conjugacy classes of all solutions of $\mathbf E_x$ which are unramified in the sense above. Let $\textrm{Sol}_{\mathbb A}(\mathbf E)$ denote the set:
$$\textrm{Sol}_{\mathbb A}(\mathbf E)=\{\prod_{x\in|F|}h_x|\textrm{$h_x\in\textrm{Sol}_{un}(\mathbf E_x)$ for almost all $x\in|F|$}\}\subseteq
\prod_{x\in|F|}\textrm{Sol}(\mathbf E_x).$$
Similarly by an ad\`elic solution of the embedding problem $\mathbf E$ we mean an expression $\prod_{x\in|F|}h_x$ such that $h_x$ is a solution of the embedding problem $\mathbf E_x$ for every $x\in|F|$ which is unramified for almost all $x$. For every $x\in|F|$ let $r_x:\textrm{Sol}(\mathbf E)\rightarrow\textrm{Sol}(\mathbf E_x)$ denote the map furnished by the rule $\widetilde{\psi}\mapsto\iota_x\circ\widetilde{\psi}$. Then image of the map
$$r=\prod_{x\in|F|}r_x:\textrm{Sol}(\mathbf E)\rightarrow\prod_{x\in|F|}
\textrm{Sol}(\mathbf E_x)$$
lies in $\textrm{Sol}_{\mathbb A}(\mathbf E)$. It is not difficult to show (see Lemma \ref{3.2} below) that the local-global principle above can be reformulated in the following equivalent form:
\begin{prin3}{} Let $F$ and $\mathbf E$ be as above. Then the set $\textrm{\rm Sol}(\mathbf E)$ is non-empty if and only if the set $\textrm{\rm Sol}_{\mathbb A}(\mathbf E)$ is non-empty.
\end{prin3}
In order to study the image of $\textrm{Sol}(\mathbf E)$ in $\textrm{Sol}_{\mathbb A}(\mathbf E)$ under the map $r$ we will define the Brauer group $\textrm{Br}(\mathbf E)$ of the embedding problem $\mathbf E$ (see Definition 2.2) and a pairing:
$$\langle\cdot,\cdot\rangle:\textrm{Sol}_{\mathbb A}(\mathbf E)\times\textrm{Br}(\mathbf E)\rightarrow\mathbb Q/\mathbb Z$$
(see Definition \ref{3.5}) such that $r(\textrm{Sol}(\mathbf E))$ is annihilated by this pairing (see Lemma \ref{3.6}). (This idea can be found already in the paper \cite{St}.) Moreover we will also define two subgroups $\textrm{Br}_1(\mathbf E)$ and $\B(\mathbf E)$ of $\textrm{Br}(\mathbf E)$ (see Definition \ref{2.4}) analogous to the algebraic Brauer group and the Brauer group of locally constant elements of algebraic varieties, respectively, playing an important role in our main results. Let $\textrm{Sol}_{\mathbb A}^{\textrm{\rm Br}}(\mathbf E)$, $\textrm{Sol}_{\mathbb A}^{\textrm{\rm Br}_1}(\mathbf E)$ and $\textrm{Sol}_{\mathbb A}^{\Bi}(\mathbf E)$ denote the subset of $\textrm{Sol}_{\mathbb A}(\mathbf E)$ annihilated by $\textrm{Br}(\mathbf E),\textrm{Br}_1(\mathbf E)$ and by $\B(\mathbf E)$ with respect to the pairing $\langle\cdot,\cdot\rangle$, respectively. Now we can state our first main result:
\begin{thm}\label{1.1} Assume that $\textrm{\rm Ker}(\mathbf E)$ is abelian and $\textrm{\rm char}(F)$ does not divide the order of $\textrm{\rm Ker}(\mathbf E)$. Then the following claims are equivalent:
$$\textrm{\rm Sol}(\mathbf E)\neq0\Leftrightarrow
\textrm{\rm Sol}_{\mathbb A}^{\textrm{\rm Br}}(\mathbf E)\neq0
\Leftrightarrow
\textrm{\rm Sol}_{\mathbb A}^{\textrm{\rm Br}_1}(\mathbf E)\neq0
\Leftrightarrow
\textrm{\rm Sol}_{\mathbb A}^{\Bi}(\mathbf E)\neq0.$$
\end{thm}
Let us make the groups mentioned above a bit more explicit. Let $\textrm{Ker}(\mathbf E)_{ab}$ denote the abelianization of $\textrm{Ker}(\mathbf E)$. Note that the natural action of $G_1$ on $\textrm{Ker}(\mathbf E)_{ab}$ via conjugation factors through $G_2$ hence $\textrm{Ker}(\mathbf E)_{ab}$ is naturally equipped with a $\Gamma$-action. Let $\textrm{\rm Ker}(\mathbf E)_{ab}^{\vee}$ denote its dual as a $\Gamma$-module. We will show that the group $\textrm{Br}(\mathbf E)$ sits in an exact sequence:
$$0\rightarrow H^1(\Gamma,\textrm{\rm Ker}(\mathbf E)_{ab}^{\vee})
\mathop{\longrightarrow}^{j_{\mathbf E}}\textrm{\rm Br}(\mathbf E)\rightarrow
H^2(\textrm{\rm Ker}(\mathbf E),\!\!\!\!\bigoplus_{p\neq\textrm{char}(F)}
\!\!\!\!\mathbb Q_p/\mathbb Z_p)^{\Gamma}
\rightarrow H^2(\Gamma,\textrm{\rm Ker}(\mathbf E)^{\vee}_{ab})$$
(see Lemma \ref{2.2}) and we will identify the subgroups $\textrm{Br}_1(\mathbf E)$ and $\B(\mathbf E)$ with the image of the group $ H^1(\Gamma,\textrm{\rm Ker}(\mathbf E)_{ab}^{\vee})$ and its Tate--Shafarevich subgroup $\Sh^1(F,\textrm{\rm Ker}(\mathbf E)_{ab}^{\vee})$ under the map $j_{\mathbf E}$, respectively (see Definition \ref{2.4} and Proposition \ref{2.7}).

In complete analogy with the situation for diophantine equations, we can also formulate a version of weak approximation, too. Obviously the map $r:\textrm{\rm Sol}(\mathbf E)\rightarrow\textrm{\rm Sol}_{\mathbb A}(\mathbf E)$ is never surjective unless $\psi:G_1\rightarrow G_2$ is a bijection. Nevertheless the situation changes if we take the Brauer-Manin obstruction into account.
\begin{thm}\label{1.2} Assume that $\textrm{\rm Ker}(\mathbf E)$ is abelian and $\textrm{\rm char}(F)$ does not divide the order of $\textrm{\rm Ker}(\mathbf E)$. Then we have:
$$r\big(\textrm{\rm Sol}(\mathbf E)\big)=
\textrm{\rm Sol}_{\mathbb A}^{\textrm{\rm Br}}(\mathbf E)=
\textrm{\rm Sol}_{\mathbb A}^{\textrm{\rm Br}_1}(\mathbf E).$$
\end{thm}
 We will give two proofs for the both theorems, at least in characteristic zero. The first proof is geometric; given an embedding problem over field $F$ we construct a connected linear algebraic group $G$ over $F$ and a homogenous space $X=X(\mathbf E)$ under $G$ over $F$ such that there is a bijection between $X(F)/G(F)$ and $\textrm{\rm Sol}(\mathbf E)$ (see Theorem \ref{homo_space}). Then we apply a result of Borovoi (see \cite{Bo1}) on the Brauer-Manin obstruction for homogeneous spaces with abelian stabilizer. Similarly we can apply a result of Colliot-Th\'el\`ene--Xu (see \cite{CX}) on strong approximation to deduce Theorem \ref{1.2}. Note that homogeneous spaces were used before to study local-global principles for split embedding problems (see \cite{DLAN} and \cite{Ha}). 
 
Since our main results are about the structure of the absolute Galois group it is more natural to have proofs which remain in this context, and do not use methods of algebraic geometry. Our second proof, which also works in positive characteristic, is of this sort, in fact it is Galois-cohomological in nature. It also has the advantage that it lead us to two auxiliary results which are interesting on their own. One of them gives an elegant description of the Tate duality pairing in terms of the Brauer-Manin pairing which we will describe next. 
\begin{defn}\label{1.3} Let $F$ be again an arbitrary field with absolute Galois group $\Gamma$ and let $\mathbf E$ be an embedding problem given by the diagram (\ref{1.0.1}). Let $\Gamma(\mathbf E)$ denote the fibre product group:
$$\Gamma(\mathbf E)=\{(a,b)\in G_1\times\Gamma|\phi(a)=\psi(b)\}
\leq G_1\times\Gamma.$$
Then $\Gamma(\mathbf E)$ sits in the exact sequence:
\begin{equation}\label{1.3.1}
\CD1@>>>\textrm{Ker}(\mathbf E)@>i_{\mathbf E}>>\Gamma(\mathbf E)
@>\pi_{\mathbf E}>>\Gamma@>>>1\endCD
\end{equation}
where the map $i_{\mathbf E}$ is given by the rule $a\mapsto(a,1)$, and the homomorphism $\pi_{\mathbf E}$ is the restriction onto $\Gamma(\mathbf E)$ of the projection of $G_1\times\Gamma$ onto the second factor. The group $\Gamma(\mathbf E)$ also inherits a topology from the product topology on $G_1\times\Gamma$ which makes $\Gamma(\mathbf E)$ a profinite group and (\ref{1.3.1}) an exact sequence in the category of Hausdorff topological groups.
\end{defn}
Let $M$ be a discrete finite abelian $\Gamma$-module. For every embedding problem $\mathbf E$ over $F$ such that $\textrm{Ker}(\mathbf E)=M$ and the set $\textrm{Sol}_{\mathbb A}(\mathbf E)$ is non-empty let $c_{\mathbf E}\in H^2(F,M)=H^2(F,\textrm{Ker}(\mathbf E))$ denote the class of the extension (\ref{1.3.1}). Note that $c_{\mathbf E}\in\Sh^2(F,M)$ since we assumed that $\textrm{Sol}_{\mathbb A}(\mathbf E)$ is non-empty. Conversely for every $c\in \Sh^2(F,M)$ there is an embedding problem $\mathbf E$ as above such that $c=c_{\mathbf E}$. Let
\begin{equation}\label{1.3.2}
\b:\Sh^1(F,M^{\vee})\times\Sh^2(F,M)\rightarrow\mathbb Q/\mathbb Z
\end{equation}
be the unique pairing such that $\b(b,c_{\mathbf E})=\langle h,b\rangle$ for every $b\in \Sh^1(F,M^{\vee})$, for every $h\in\textrm{Sol}_{\mathbb A}(\mathbf E)$, and embedding problem $\mathbf E$ as above. Note that $\b_{\mathbf E}(b)=\langle h,b\rangle$ is independent of the choice of $h$, and since for every $b$ as above the value of $\b_{\mathbf E}(b)$ only depends on the isomorphism class of the embedding problem $\mathbf E$ (see Definition \ref{4.8} and the proof of Lemma \ref{4.6} below), the pairing $\b$ is well-defined. Assume now that $\textrm{\rm char}(F)$ does not divide the order of $M$ and let
$$\tau:\Sh^1(F,M^{\vee})\times\Sh^2(F,M)\longrightarrow
\mathbb Q/\mathbb Z$$
denote the Tate duality pairing.
\begin{thm}\label{1.4} We have $\b=-\tau$.
\end{thm}
In the paper \cite{HaSz} Harari and Szamuely gave a geometric interpretation of the duality pairing. Our result offers another geometric interpretation by relating it to the Brauer--Manin pairing. This result is the key ingredient of the cohomological proof of Theorem \ref{1.1}. The other auxiliary result which we mentioned above is a result on the cup product (Theorem \ref{6.3}) which plays a central role in the cohomological proof of Theorem \ref{1.2} through its group theoretical counterpart (Corollary \ref{6.5}). In the last chapter we also present an example (Example \ref{the_example}) which shows that the algebraic Brauer--Manin obstruction is not sufficient for weak approximation for embedding problems with non-abelian kernel. However it remains a very interesting challenge to construct an embedding problem where the failure of the local-global principle cannot be explained by the Brauer-Manin obstruction.
\begin{cont} In the next chapter we will define the Brauer group of embedding problems and study its structure. We will introduce the analogue of the Brauer-Manin obstruction in the third chapter. In the fourth chapter we show that the analogue of the algebraic Brauer-Manin obstruction is equivalent to the analogue of the abelian descent obstruction. This result and the closely related Theorem \ref{4.9} are relatively straightforward consequences of Theorems \ref{1.1} and \ref{1.2} once a suitable formalism is set up. We prove a local-global principle for cohomology classes using Poitou-Tate duality in the fifth chapter. In the sixth chapter we prove a theorem on the cup product in topology which we use to deduce a similar statement in group cohomology. We compare the Tate duality pairing with the Brauer--Manin paring in the seventh chapter. With the help of these results we prove Theorems \ref{1.1} and \ref{1.2} in the eighth chapter. We present a geometric construction and use it to give another proof of Theorem \ref{1.1} in the ninth chapter. In the last chapter we present a counter-example to weak approximation not explained by the algebraic Brauer-Manin obstruction.
\end{cont} 
\begin{ackn} The first author was partially supported by the EPSRC grants P19164 and P36794. The second author was partially supported by a Clore Fellowship.
\end{ackn} 

\section{The Brauer group of embedding problems}

\begin{defn}\label{2.1} Via the homomorphism $\pi_{\mathbf E}$ we may consider every discrete $\Gamma$-module $M$ a discrete $\Gamma(\mathbf E)$-module, too, which will be denoted also by $M$ by slight abuse of notation. For every $n\in\mathbb N$ and abelian group $M$ let $M[n]$ denote the $n$-torsion of $M$ and for every such $M$ let $M^{ct}$ denote the quotient of $M$ by its torsion. Let $\textrm{\textrm{Br}}(\mathbf E)$ denote the cokernel of the homomorphism
$$\pi_{\mathbf E}^*:H^2(\Gamma,\overline F^*)\rightarrow H^2(\Gamma(\mathbf E),\overline F^*).$$
For every finite discrete abelian $\Gamma$-module $M$ let $M^{\vee}$ denote the dual of $M$:
$$M^{\vee}=\textrm{Hom}(M,\overline F^*).$$
\end{defn}
Assume now that $F$ is either a local or a global field and continue to use the notation which we introduced above.
\begin{lemma}\label{2.2} We have the following short exact sequence:
$$0\rightarrow H^1(\Gamma,\textrm{\rm Ker}(\mathbf E)_{ab}^{\vee})
\mathop{\longrightarrow}^{j_{\mathbf E}}\textrm{\rm Br}(\mathbf E)\rightarrow
H^2(\textrm{\rm Ker}(\mathbf E),\!\!\!\!\bigoplus_{p\neq\textrm{\rm char}(F)}
\!\!\!\!\mathbb Q_p/\mathbb Z_p)^{\Gamma}
\rightarrow H^2(\Gamma,\textrm{\rm Ker}(\mathbf E)^{\vee}_{ab}),$$
where we equip $\bigoplus_{p\neq\textrm{\rm char}(F)}\mathbb Q_p/\mathbb Z_p$ with the trivial $\textrm{\rm Ker}(\mathbf E)$-action.
\end{lemma}
\begin{proof} The Hochschild-Serre spectral sequence:
\begin{equation}\label{2.2.1}
E^2_{p,q}=H^p(\Gamma,H^q(\textrm{Ker}(\mathbf E),\overline F^*))\Rightarrow
H^{p+q}(\Gamma(\mathbf E),\overline F^*)
\end{equation}
furnishes on $H^2(\Gamma(\mathbf E),\overline F^*)$ a filtration:
\begin{equation}\label{2.2.2}
0=E^2_3\subseteq E^2_2
\subseteq E^2_1\subseteq
E^2_0=H^2(\Gamma(\mathbf E),\overline F^*)
\end{equation}
such that
$$E^{\infty}_{p,2-p}\cong E^2_p/E^2_{p+1},\quad p=0,1,2.$$
Because $H^3(\Gamma,\overline F^*)=0$ (see Proposition 15 of \cite{Se} on page 93 when $F$ is a local field, and see Corollary 4.21 of \cite{Mi}, page 80 when $F$ is a global field), the coboundary map:
$$d^2_{1,1}:E^2_{1,1}=
H^1(\Gamma,H^1(\textrm{\rm Ker}(\mathbf E),\overline F^*))=
H^1(\Gamma,\textrm{\rm Ker}(\mathbf E)^{\vee}_{ab})
\rightarrow
H^3(\Gamma,\overline F^*)$$
is zero and therefore:
$$E^{\infty}_{1,1}=E^3_{1,1}=\textrm{Ker}(d^2_{1,1})=H^1(\Gamma,\textrm{\rm Ker}(\mathbf E)^{\vee}_{ab}).$$
We have the following short exact sequence of trivial Ker$(\mathbf E)$-modules:
$$0\longrightarrow\bigoplus_{p\neq\textrm{\rm char}(F)}\mathbb Q_p/\mathbb Z_p\longrightarrow^{\!\!\!\!\!\!\!\!\!u}\ \ \overline F^*
\longrightarrow(\overline F^*)^{ct}\longrightarrow0$$
where the image of the map $u$ is the module of roots of unity. The Ker$(\mathbf E)$-module $(\overline F^*)^{ct}$ is a vector space over $\mathbb Q$ so its higher cohomology groups vanish. Hence $H^k(\textrm{Ker}(\mathbf E),\bigoplus_{p\neq\textrm{\rm char}(F)}\mathbb Q_p/\mathbb Z_p)=H^k(\textrm{Ker}(\mathbf E),\overline F^*)$ for every positive integer $k$. Therefore
$$E^{\infty}_{0,2}=E^2_{0,2}=\textrm{Ker}(d^2_{0,2}),$$
where the coboundary map $d^2_{0,2}$ is a homomorphism:
$$E^2_{0,2}=H^2(\textrm{\rm Ker}(\mathbf E),\!\!\!\!
\bigoplus_{p\neq\textrm{\rm char}(F)}
\!\!\!\!\mathbb Q_p/\mathbb Z_p)^{\Gamma}
\rightarrow H^2(\Gamma,H^1(\textrm{\rm Ker}(\mathbf E),\overline F^*))=
H^2(\Gamma,\textrm{\rm Ker}(\mathbf E)^{\vee}_{ab}).$$
Because
$$E^{\infty}_{2,0}=E^3_{2,0}=\pi_{\mathbf E}^*(H^2(\Gamma,\overline F^*))
\subseteq H^2(\Gamma(\mathbf E),\overline F^*),$$
the claim is now clear.
\end{proof}
Assume now that $F$ is a global field. Note that for every $x \in |F|$ the following diagram:
$$\xymatrix{
H^2(\Gamma,\overline F^*)\ar[r]^{\iota_x^*}\ar[d]^{\pi_{\mathbf E}^*}
& H^2(\Gamma_x,\overline F^*)\ar[r]^{(\eta_x)_*}\ar[d]^{\pi_{\mathbf E_x}^*}
& H^2(\Gamma_x,\overline F_x^*)\ar[d]^{\pi_{\mathbf E_x}^*}\\
H^2(\Gamma(\mathbf E),\overline F^*)\ar[r]^{(\textrm{id}_{G_1}\times\iota_x)^*}
& H^2(\Gamma_x(\mathbf E_x),\overline F^*)\ar[r]^{(\eta_x)_*}&
H^2(\Gamma_x(\mathbf E_x),\overline F_x^*)}$$
is commutative, where the maps on the left are restriction maps in group cohomology, and hence it gives rise to a map:
$$j_x:\textrm{Br}(\mathbf E)\longrightarrow\textrm{Br}(\mathbf E_x).$$
\begin{defn}\label{2.3} Let $\B(\mathbf E)$ denote the intersection:
$$\B(\mathbf E)=\bigcap_{x\in|F|}\textrm{Ker}(j_x)\leq\textrm{Br}(\mathbf E).$$
Let $\textrm{Br}_1(\mathbf E)$ denote the kernel of the homomorphism:
$$\textrm{\rm Br}(\mathbf E)\longrightarrow
H^2(\textrm{\rm Ker}(\mathbf E),\bigoplus_{p\neq\textrm{\rm char}(F)}\mathbb Q_p/\mathbb Z_p)^{\Gamma}.$$
\end{defn}
\begin{lemma}\label{2.4} The group $\B(\mathbf E)$ is a subgroup of $\textrm{\rm Br}_1(\mathbf E)$.
\end{lemma}
\begin{proof} As we saw in the proof above the Hochschild-Serre spectral sequence
$$F^2_{p,q}=H^p(\Gamma_x,H^q(\textrm{Ker}(\mathbf E),\overline F_x^*))\Rightarrow
H^{p+q}(\Gamma_x(\mathbf E_x),\overline F_x^*),$$
furnishes on $H^2(\Gamma_x(\mathbf E_x),\overline F_x^*)$ a filtration:
\begin{equation}\label{2.4.1}
0=F^2_3\subseteq F^2_2
\subseteq F^2_1\subseteq
F^2_0=H^2(\Gamma_x(\mathbf E_x),\overline F_x^*)
\end{equation}
such that
$$F^{\infty}_{p,2-p}\cong F^2_p/F^2_{p+1},\quad p=0,1,2.$$
The homomorphism $(\textrm{id}_{G_1}\times\iota_x)^*\circ(\eta_x)_*:
H^2(\Gamma(\mathbf E),\overline F^*)
\rightarrow H^2(\Gamma_x(\mathbf E_x),\overline F_x^*)$ respects the filtrations (\ref{2.2.2}) and (\ref{2.4.1}) in the sense that $(\textrm{id}_{G_1}\times\iota_x)^*\circ(\eta_x)_*(E^2_p)\subseteq F^2_p$ for every $p=0,1,2$. Moreover the homomorphism:
$$E^{\infty}_{0,2}\cong E^2_0/E^2_1
\longrightarrow F^2_0/F^2_1\cong F^{\infty}_{0,2}$$
induced by $(\textrm{id}_{G_1}\times\iota_x)^*\circ(\eta_x)_*$ is the restriction of the map:
\begin{equation}\label{2.4.2}
H^2(\textrm{\rm Ker}(\mathbf E),
\!\!\!\!\bigoplus_{p\neq\textrm{\rm char}(F)}
\!\!\!\!\mathbb Q_p/\mathbb Z_p)^{\Gamma}
\longrightarrow H^2(\textrm{\rm Ker}(\mathbf E),
\!\!\!\!\bigoplus_{p\neq\textrm{\rm char}(F)}
\!\!\!\!\mathbb Q_p/\mathbb Z_p
)^{\Gamma_x}
\end{equation}
onto the kernel of the homomorphism
$$d^2_{0,2}:H^2(\textrm{\rm Ker}(\mathbf E),\!\!\!\!\bigoplus_{p\neq\textrm{\rm char}(F)}
\!\!\!\!\mathbb Q_p/\mathbb Z_p)^{\Gamma}
\rightarrow
H^2(\Gamma,\textrm{\rm Ker}(\mathbf E)^{\vee}).$$
Because the homomorphism in (\ref{2.4.2}) is injective, the claim is now clear.
\end{proof}
\begin{lemma}\label{2.5} The group $\textrm{\rm Br}_1(\mathbf E)$ is annihilated by $|\textrm{\rm Ker}(\mathbf E)|$. The group $\textrm{\rm Br}(\mathbf E)$ is annihilated by $|\textrm{\rm Ker}(\mathbf E)|^2$.
\end{lemma}
\begin{proof} The group $\textrm{\rm Br}_1(\mathbf E)$ is isomorphic to the cohomology group $H^1(\Gamma,\textrm{\rm Ker}(\mathbf E)_{ab}^{\vee})$ which is annihilated by $|\textrm{\rm Ker}(\mathbf E)|$ since the $\Gamma$-module $\textrm{\rm Ker}(\mathbf E)_{ab}^{\vee}$ is annihilated by by $|\textrm{\rm Ker}(\mathbf E)|$. For the second claim it will be enough to show that the quotient $\textrm{\rm Br}(\mathbf E)/\textrm{\rm Br}_1(\mathbf E)$ is annihilated by $|\textrm{\rm Ker}(\mathbf E)|$. This group is isomorphic to a subgroup of $H^2(\textrm{\rm Ker}(\mathbf E),\bigoplus_{p\neq\textrm{\rm char}(F)}\mathbb Q_p/\mathbb Z_p)$. For every $\textrm{\rm Ker}(\mathbf E)$-module $M$ the cohomology group $H^2(\textrm{\rm Ker}(\mathbf E),M)$ is annihilated by the order of the finite group $\textrm{\rm Ker}(\mathbf E)$. The claim is now clear.
\end{proof}
\begin{notn} For every $k\in\mathbb N$ and for finite discrete $\Gamma$-module $M$ let $\Sh^k(F,M)$ denote the subgroup:
$$\Sh^k(F,M)=\textrm{Ker}\left(
\prod_{x\in|F|}\iota_x^*:
H^k(\Gamma,M)\rightarrow
H^k(\Gamma_x,M)\right)\leq H^k(\Gamma,M).$$
\end{notn}
\begin{prop}\label{2.7} We have:
$$\B(\mathbf E)=\Sh^1(F,\textrm{\rm Ker}(\mathbf E)^{\vee}_{ab}).$$
\end{prop}
\begin{proof} As we already noted for every $x\in|F|$ the homomorphism
$$(\textrm{id}_{G_1}\times\iota_x)^*\circ(\eta_x)_*:
H^2(\Gamma(\mathbf E),\overline F^*)
\rightarrow H^2(\Gamma_x(\mathbf E_x),\overline F_x^*)$$
respects the filtrations (\ref{2.2.2}) and (\ref{2.4.1}), that is: $(\textrm{id}_{G_1}\times\iota_x)^*\circ(\eta_x)_*(E^2_p)\subseteq F^2_p$ for every $p=0,1,2$. Moreover the homomorphism:
$$E^{\infty}_{1,1}\cong E^2_1/E^2_2
\longrightarrow F^2_1/F^2_2\cong F^{\infty}_{1,1}$$
induced by $(\textrm{id}_{G_1}\times\iota_x)^*\circ(\eta_x)_*$ is the restriction homomorphism:
$$\iota_x^*:E^{\infty}_{1,1}=H^1(\Gamma,\textrm{\rm Ker}(\mathbf E)_{ab}^{\vee})
\longrightarrow
H^1(\Gamma_x,\textrm{\rm Ker}(\mathbf E)_{ab}^{\vee})
=F^{\infty}_{1,1}$$
(where we used that $H^3(\Gamma_x,\overline F_x^*)=0$). Hence under the isomorphism $\textrm{Br}_1(\mathbf E)\cong H^1(\Gamma,\textrm{\rm Ker}(\mathbf E)_{ab}^{\vee})$ furnished by the spectral sequence
(\ref{2.2.1}) the subgroup $\B(\mathbf E)$ is identified with the intersection of the kernels of the maps $\iota_x^*$ for all $x\in|F|$. The claim is now clear.
\end{proof}
\begin{cor} The group $\B(\mathbf E)$ is finite.
\end{cor}
\begin{proof} The group $\Sh^1(F,\textrm{\rm Ker}(\mathbf E)^{\vee}_{ab})$ is known to be finite (see Theorem 4.10 of \cite{Mi}, page 70). The claim now follows from the proposition above.
\end{proof}

\section{The Brauer-Manin obstruction}

\begin{notn}\label{3.1} Let $F$ be again an arbitrary field with absolute Galois group $\Gamma$ and let $\mathbf E$ be an embedding problem given by the diagram (\ref{1.0.1}). Note that the map which assigns to every solution $h$ of $\mathbf E$ the homomorphism
$$s(h):\Gamma\rightarrow\Gamma(\mathbf E)\subseteq
 G_1\times\Gamma$$
given by the rule $g\mapsto(h(g),g)$ is a bijection between the solutions of $\mathbf E$ and continuous sections of the exact sequence (\ref{1.3.1}). This bijection induces a bijection between $\textrm{\rm Sol}(\mathbf E)$ and the conjugacy classes of sections of (\ref{1.3.1}). We will always identify these two pairs of sets under these bijections.
\end{notn}
Assume now that $F$ is a global field and let $\mathbf E$ be as above.
\begin{lemma}\label{3.2} The set $\textrm{\rm Sol}_{\mathbb A}(\mathbf E)$ is non-empty if and only if for every $x\in|F|$ there is no local obstruction to the embedding problem $\mathbf E$ at the place $x$.
\end{lemma}
\begin{proof} Note that for almost all $x\in|F|$ the homomorphism $\psi\circ\iota_x$ is unramified and for every such $x$ the set $\textrm{Sol}_{un}(\mathbf E_x)$ is non-empty since we can lift any continuous homomorphism $\widehat{\mathbb Z}\rightarrow G_2$ to a continuous homomorphism
$\widehat{\mathbb Z}\rightarrow G_1$ with respect to the surjective map $\phi$.
\end{proof}
For every $x\in|F|$ let
$$\textrm{inv}_x:\textrm{Br}(F_x)=H^2(\Gamma_x,\overline F_x^*)
\rightarrow\mathbb Q/\mathbb Z$$
denote the canonical invariant of the Brauer group Br$(F_x)$ of the local field $F_x$. Let Inf denote the inflation map in group cohomology, as usual.
\begin{lemma}\label{3.3} For every $c\in H^2(\Gamma(\mathbf E),\overline F^*)$ and for every ad\`elic solution $\prod_{x\in|F|}h_x$ of $\mathbf E$ the image of $c$ under the composition:
$$\CD
H^2(\Gamma(\mathbf E),\overline F^*)@>{(\textrm{\rm id}_{G_1}\times\iota_x)^*}>>H^2(\Gamma_x(\mathbf E_x),\overline F^*)
@>{(\eta_x)_*}>>H^2(\Gamma_x(\mathbf E_x),\overline F_x^*)
@>s(h_x)^*>>\endCD$$
$$\CD H^2(\Gamma_x,\overline F_x^*)
@>\textrm{\rm inv}_x>>\mathbb Q/\mathbb Z
\endCD$$
is zero for almost all $x\in|F|$.
\end{lemma}
\begin{proof} Note that there is an open normal subgroup $U\triangleleft\Gamma(\mathbf E)$ such that $c$ is the image of a cohomology class $\widetilde c\in H^2(\Gamma(\mathbf E)/U,(\overline F^*)^U)$ with respect to the inflation map $H^2(\Gamma(\mathbf E)/U,(\overline F^*)^U)\rightarrow H^2(\Gamma(\mathbf E),\overline F^*)$. Let $G=\Gamma(\mathbf E)/U$ be the quotient by $U$. Moreover let $\rho:\Gamma(\mathbf E)\rightarrow G$ denote the quotient map and let $\overline{\rho}:\Gamma\rightarrow G/\rho(\textrm{Ker}(\mathbf E))$ be the unique continuous homomorphism such that $\overline{\rho}\circ\pi_{\mathbf E}$ is the composition of $\rho$ and the quotient map $G\rightarrow G/\rho(\textrm{Ker}(\mathbf E))$. We may assume without the loss of generality that $\overline{\rho}\circ\iota_x$ and $h_x$ are unramified. In this the case the homomorphism $\rho\circ(\textrm{\rm id}_{G_1}\times\iota_x)\circ s(h_x)$ is also unramified. Therefore the cohomology class
$$\widetilde c_x\stackrel{\textrm{def}}{=}
(\eta_x)_*(\rho\circ(\textrm{\rm id}_{G_1}\times\iota_x)\circ s(h_x))^*(\widetilde c)\in H^2(\Gamma_x,\overline F_x^*)$$
lies in the image of the inflation map $H^2(\widehat{\mathbb Z},\mathbb (\overline F_x^*)^{I_x})\rightarrow H^2(\Gamma_x,\overline F_x^*)$ where $I_x\triangleleft\Gamma_x$ is the inertia subgroup. Note that for all but finitely many $x$ the image of $\widetilde c_x$ under the map $H^2(\widehat{\mathbb Z},\mathbb (\overline F_x^*)^{I_x})\rightarrow H^2(\widehat{\mathbb Z},\mathbb Z)$ induced by $x$ is zero (for example because the valuation of the values of a fixed cocycle representing $\widetilde c$ with respect to $x$ is zero for all but finitely many $x$). Because the map $H^2(\widehat{\mathbb Z},\mathbb (\overline F_x^*)^{I_x})\rightarrow H^2(\widehat{\mathbb Z},\mathbb Z)$ is an isomorphism (see Theorem 2 on page 130 in \cite{Se0}), the claim is now clear.
\end{proof}
Note that for every $x\in|F|$ and $h_x$ as above the map $s(h_x)^*$ only depends on the conjugacy class of $h_x$, therefore the pairing:
$$(\cdot,\cdot):\textrm{Sol}_{\mathbb A}(\mathbf E)\times
H^2(\Gamma(\mathbf E),\overline F^*)\rightarrow\mathbb Q/\mathbb Z$$
given by the rule
$$(\prod_{x\in|F|}h_x,c)=\sum_{x\in|F|}
\textrm{\rm inv}_x(s(h_x)^*((\eta_x)_*(
(\textrm{id}_{G_1}\times\iota_x)^*( c))))$$
is well-defined, because all but finitely many of the summands are zero by the lemma above.
\begin{lemma}\label{3.4} The image of group $H^2(\Gamma,\overline F^*)$ with respect to the homomorphism $\pi_{\mathbf E}^*$ is annihilated by the pairing $(\cdot,\cdot)$.
\end{lemma}
\begin{proof} Note that for every $x\in|F|$ and for every section $s:\Gamma_x\rightarrow\Gamma_x(\mathbf E_x)$ of the short exact sequence (\ref{1.3.1}) we have $\pi_E\circ(\textrm{\rm id}_{G_1}\times\iota_x)\circ s=\iota_x$. Hence for every $c\in H^2(\Gamma,\overline F^*)$ and  every ad\`elic solution $h=\prod_{x\in|F|}h_x$ of $\mathbf E$ we have:
\begin{eqnarray}
(h,\pi_{\mathbf E}^*( c ))&=&\sum_{x\in|F|}
\textrm{\rm inv}_x(s(h_x)^*((\eta_x)_*(
(\textrm{id}_{G_1}\times\iota_x)^*( \pi_{\mathbf E}^*(c )))))
\nonumber\\
&=&
\sum_{x\in|F|}
\textrm{\rm inv}_x((\eta_x)_*\iota_x^*(c))=0\nonumber
\end{eqnarray}
by the reciprocity law for Brauer groups over global fields.
\end{proof}
\begin{defn}\label{3.5} By the lemma above we have a pairing:
$$\langle\cdot,\cdot\rangle:\textrm{Sol}_{\mathbb A}(\mathbf E)\times\textrm{\textrm{Br}}(\mathbf E)\rightarrow\mathbb Q/\mathbb Z$$
such that for every $h\in\textrm{Sol}_{\mathbb A}(\mathbf E)$ and $c\in H^2(\Gamma(\mathbf E),\overline F^*)$ we have:
$$(h,c)=\langle h,\sigma_{\mathbf E}(h)\rangle$$
where $\sigma_{\mathbf E}:H^2(\Gamma(\mathbf E),\overline F^*)\rightarrow\textrm{\textrm{Br}}(\mathbf E)$ is the tautological surjection. For every subset $X\subseteq\textrm{Br}(\mathbf E)$ let $\textrm{Sol}_{\mathbb A}^X(\mathbf E)$ denote the set:
$$\textrm{Sol}_{\mathbb A}^X(\mathbf E)=\{h\in\textrm{\rm Sol}_{\mathbb A}(\mathbf E)|\langle h,c\rangle=0\
(\forall c\in X)\}.$$
In the special case when $X=\textrm{Br}(\mathbf E),\textrm{Br}_1(\mathbf E)$ or $\B(\mathbf E)$ we will use the shorter superscripts $\textrm{Br},\textrm{Br}_1$ and $\B$, respectively. Clearly we have the inclusions:
$$\textrm{Sol}_{\mathbb A}^{\textrm{Br}}(\mathbf E)\subseteq\textrm{Sol}_{\mathbb A}^{\textrm{Br}_1}(\mathbf E)\subseteq
\textrm{Sol}_{\mathbb A}^{\Bi}(\mathbf E)\subseteq\textrm{Sol}_{\mathbb A}(\mathbf E).$$
\end{defn}
\begin{lemma}\label{3.6} We have $r(\textrm{\rm Sol}(\mathbf E))\subseteq
\textrm{\rm Sol}_{\mathbb A}^{\textrm{\rm Br}}(\mathbf E)$.
\end{lemma}
\begin{proof} Let $h$ be a solution of $\mathbf E$. Then for every $x\in|F|$ we have
$$(\textrm{id}_{G_1}\times\iota_x)\circ s(\iota_x\circ h)
= s(h)\circ\iota_x.$$
Hence for every $c\in H^2(\Gamma(\mathbf E),\overline F^*)$ we have:
\begin{eqnarray}
(h,c)&=&\sum_{x\in|F|}
\textrm{\rm inv}_x(s(\iota_x\circ h)^*((\eta_x)_*(
(\textrm{id}_{G_1}\times\iota_x)^*( c))))\nonumber
\\
&=&
\sum_{x\in|F|}\textrm{\rm inv}_x((\eta_x)_*\iota_x^*(s(h)^*( c)))=0
\nonumber
\end{eqnarray}
by the reciprocity law for Brauer groups over global fields applied to the cohomology class $s(h)^*(c)\in H^2(\Gamma,\overline F^*)$.
\end{proof}
Now we can formulate an improved version of the local-global principle for embedding problems:
\begin{prin} The set $\textrm{\rm Sol}(\mathbf E)$ is non-empty if and only if the set $\textrm{\rm Sol}^{\textrm{\rm Br}}_{\mathbb A}(\mathbf E)$ is non-empty.
\end{prin}
We will say that the Brauer-Manin obstruction is the only one to the local-global principle for $\mathbf E$ if the claim above is true for $\mathbf E$. Theorem \ref{1.1} implies that this is the case when Ker$(\mathbf E)$ is abelian and its order is not divisible by char$(F)$.
\begin{defn}\label{4.8} Let $\mathbf E$ be an embedding problem over $F$ such that the set $\textrm{Sol}_{\mathbb A}(\mathbf E)$ is non-empty and let $b$ be an element of $\B(\mathbf E)$. Choose a $b'\in H^2(\Gamma(\mathbf E),\overline F^*)$ such that $\sigma_{\mathbf E}(b')=b$. By definition for every $x\in|F|$ there is a $b_x\in H^2(\Gamma_x,\overline F_x^*)$ such that $(\eta_x)_*(\textrm{id}_{G_1}\times\iota_x)^*(b')=\pi_{\mathbf E_x}^*(b_x)$. Therefore for every solution $h_x$ of $\mathbf E_x$ we have:
$$s(h_x)^*((\eta_x)_*(
(\textrm{id}_{G_1}\times\iota_x)^*( b')))=
s(h_x)^*(\pi_{\mathbf E_x}^*(b_x))=b_x,$$
and hence the value of
$$\b_{\mathbf E}(b)=\langle h,b\rangle=\sum_{x\in|F|}\textrm{inv}_x(b_x)$$
does not depend on the choice of the ad\`elic solution $h=\prod_{x\in|F|}h_x$ of $\mathbf E$. Let
$$\b_{\mathbf E}:\B(\mathbf E)=\Sh^1(F,\textrm{Ker}(\mathbf E)^{\vee}_{ab})\rightarrow\mathbb Q/\mathbb Z$$
denote the function defined by the formula above.
\end{defn}
\begin{rem}\label{3.8} Note that for every $\prod_{x\in|F|}h_x\in\prod_{x\in|F|}
\textrm{Sol}(\mathbf E_x)$ and every $b$, $b'$ and $b_x$ as above the infinite sum
$$\sum_{x\in|F|}\textrm{\rm inv}_x(s(h_x)^*((\eta_x)_*(
(\textrm{id}_{G_1}\times\iota_x)^*( b'))))=\sum_{x\in|F|}\textrm{inv}_x(b_x)$$
is actually finite, and it is equal to the value of $\b_{\mathbf E}(b)$. We will use this observation in the geometric proof of Theorem \ref{1.1} in section 9.
\end{rem}

\section{Descent for embedding problems}

\begin{defn} Let $\mathbf E$ and $\mathbf E'$ be two embedding problems over an arbitrary field $F$ given by the diagrams
$$\CD @.\Gamma
@.\quad\quad\quad\quad@.
@.\Gamma\\
@.@V\psi VV\textrm{and}@.
@.@V\psi' VV\\
G_1@>\phi>>G_2
@.\quad\quad\quad\quad@.
G_1'@>\phi'>>G_2',\endCD$$
respectively. A map $g:\mathbf E\rightarrow\mathbf E'$ from the embedding problem $\mathbf E$ and to the embedding problem $\mathbf E'$ is a pair of group homomorphisms $g_1:G_1\rightarrow G_1'$ and $g_2:G_2\rightarrow G_2'$ such that the diagrams
$$\CD G_1@>\phi>>G_2@.
\quad\quad\quad\quad@.
\Gamma@=\Gamma\\
@Vg_1VV@Vg_2VV\textrm{and}@.
@VV\psi V@V\psi' VV\\
G_1'@>\phi'>>G_2'@.
\quad\quad\quad\quad@.
G_2@>g_2>>G_2',\endCD$$
are commutative. Let $\textrm{Prob}(F)$ denote the category whose objects are embedding problems over $F$ and whose morphisms are maps between them as defined above.
\end{defn}
\begin{defn} Let Sol denote the functor from $\textrm{Prob}(F)$ to the category of sets which maps every object $\mathbf E$ of $\textrm{Prob}(F)$ to Sol$(\mathbf E)$ and every morphism $g:\mathbf E\rightarrow\mathbf E'$ to the map Sol$(g):\textrm{Sol}(\mathbf E)\rightarrow\textrm{Sol}(\mathbf E')$ furnished by the rule $\widetilde{\psi}\mapsto g_1\circ\widetilde{\psi}$ where $g_1$ is as above. There is a similar functor $\textrm{Sol}_{\mathbb A}$ from $\textrm{Prob}(F)$ to the category of sets that maps every object $\mathbf E$ of $\textrm{Prob}(F)$ to $\textrm{Sol}_{\mathbb A}(\mathbf E)$ when $F$ is a global field. Moreover there is a unique functor $\Gamma(\cdot)$ from $\textrm{Prob}(F)$ to the category of profinite groups which maps every object $\mathbf E$ of $\textrm{Prob}(F)$ to $\Gamma(\mathbf E)$ and for every morphism $g:\mathbf E\rightarrow\mathbf E'$ the diagram:
$$\CD\Gamma(\mathbf E)@>>>G_1\times\Gamma\\
@V\Gamma(g)VV@V{g_1\times\textrm {id}_{\Gamma}}VV\\
\Gamma(\mathbf E')@>>>G_1'\times\Gamma\endCD$$
commutes.
\end{defn}
\begin{lemma} For every morphism $g:\mathbf E\rightarrow\mathbf E'$ in $\textrm{\rm Prob}(F)$ the homomorphism:
$$\Gamma(g)^*:H^2(\Gamma(\mathbf E'),\overline F^*)\longrightarrow
H^2(\Gamma(\mathbf E),\overline F^*)$$
respects the filtrations on $H^2(\Gamma(\mathbf E'),\overline F^*)$ and $
H^2(\Gamma(\mathbf E),\overline F^*)$ induced by the Hoch\-schild-Serre spectral sequence (\ref{2.2.1}) for $\mathbf E'$ and $\mathbf E$, respectively.
\end{lemma}
\begin{proof} Note that the diagram:
$$\CD1@>>>\textrm{Ker}(\mathbf E)@>i_{\mathbf E}>>\Gamma(\mathbf E)
@>\pi_{\mathbf E}>>\Gamma@>>>1\\
@.@V{g_1|_{\textrm{Ker}(\mathbf E)}}VV@V{\Gamma(g)}VV@|@.\\
1@>>>\textrm{Ker}(\mathbf E')@>i_{\mathbf E'}
>>\Gamma(\mathbf E')
@>\pi_{\mathbf E'}>>\Gamma@>>>1\endCD$$
commutes. The claim is now clear.
\end{proof}
\begin{notn} By the above there is a unique functor from $\textrm{Prob}(F)$ to the category of abelian groups which will be denoted by Br and which maps every object $\mathbf E$ of $\textrm{Prob}(F)$ to $\textrm{Br}(\mathbf E)$ and for every morphism $g:\mathbf E\rightarrow\mathbf E'$ the diagram:
$$\CD H^2(\Gamma(\mathbf E'),\overline F^*)
@>\Gamma(g)^*>>
H^2(\Gamma(\mathbf E),\overline F^*)\\
@V\sigma_{\mathbf E'}VV@V\sigma_{\mathbf E}VV\\
\textrm{\rm Br}(\mathbf E')@>\textrm{\rm Br}(g)>>
\textrm{\rm Br}(\mathbf E)\endCD$$
is commutative (where the vertical maps were introduced in Definition \ref{3.5}).
\end{notn}
Assume now that $F$ is either a local or a global field. An immediate consequence of the lemma above is the following
\begin{cor}\label{4.6} For every morphism $g:\mathbf E\rightarrow\mathbf E'$ in $\textrm{\rm Prob}(F)$ the diagram:
$$\CD H^1(\Gamma,\textrm{\rm Ker}(\mathbf E')_{ab}^{\vee})
@>{(g^{ab}_1)^{\vee}_*}>>
H^1(\Gamma,\textrm{\rm Ker}(\mathbf E)_{ab}^{\vee})\\
@VVV@VVV\\
\textrm{\rm Br}_1(\mathbf E')@>\textrm{\rm Br}(g)|_{\textrm{\rm Br}_1(\mathbf E')}>>
\textrm{\rm Br}_1(\mathbf E)\endCD$$
is commutative, where $(g_1^{ab})^{\vee}:\textrm{\rm Ker}(\mathbf E')_{ab}^{\vee}\rightarrow \textrm{\rm Ker}(\mathbf E)_{ab}^{\vee}$ is the dual of the abelianization of $g_1$ and the vertical arrows are furnished by the exact sequence in Lemma \ref{2.2}.\qed
\end{cor}
Assume now that $F$ is a global field.
\begin{lemma}\label{4.7} For every morphism $g:\mathbf E\rightarrow\mathbf E'$ in $\textrm{\rm Prob}(F)$ and for every subset $X\subseteq\textrm{\rm Br}(\mathbf E)$ we have:
$$\textrm{\rm Sol}_{\mathbb A}(g)^{-1}\big(\textrm{\rm Sol}_{\mathbb A}^X(\mathbf E')\big)=
\textrm{\rm Sol}_{\mathbb A}^{\textrm{\rm Br}(g)(X)}(\mathbf E).$$
\end{lemma}
\begin{proof} Since for every $h\in\textrm{\rm Sol}_{\mathbb A}(\mathbf E)$ and $c\in H^2(\Gamma(\mathbf E'),\overline F^*)$ we have:
$$(h,\Gamma(g)^*(c))=(\textrm{Sol}_{\mathbb A}(g)(h),c)$$
the claim is clear.
\end{proof}
\begin{defn}\label{4.5} Let $\mathbf E^{ab}$ be the embedding problem given by the diagram:
$$\CD @.\Gamma\\
@.@V\psi VV\\
G_1/[\textrm{Ker}(\mathbf E),\textrm{Ker}(\mathbf E)]
@>\overline{\phi}>>G_2\endCD$$
where $\overline{\phi}$ is the unique homomorphism such that the composition of  the quotient map $e:G_1\rightarrow G_1/[\textrm{Ker}(\mathbf E),\textrm{Ker}(\mathbf E)]$ and $\overline{\phi}$ is $\phi$ and we assume that $\mathbf E$ was given by the diagram (\ref{1.0.1}). 
\end{defn}
\begin{thm}\label{4.9} Assume that $\textrm{\rm char}(F)$ does not divide the order of $\textrm{\rm Ker}(\mathbf E)_{ab}$. Then the following claims are equivalent:
$$\textrm{\rm Sol}(\mathbf E^{ab})\neq0\textrm{ and }
\textrm{\rm Sol}_{\mathbb A}(\mathbf E)\neq0
\Leftrightarrow
\textrm{\rm Sol}_{\mathbb A}^{\Bi}(\mathbf E)\neq0.$$
\end{thm}
\begin{proof} Let $\mathbf e:\mathbf E\rightarrow\mathbf E^{ab}$ be the map given by the homomorphism $e:G_1\rightarrow G_1/[\textrm{Ker}(\mathbf E),\textrm{Ker}(\mathbf E)]$ (introduced in Definition \ref{4.5}) and the identity map $G_2\rightarrow G_2$. By Corollary \ref{4.6} the map $\textrm{\rm Br}(\mathbf e)|_{\Bi(\mathbf E^{ab})}:\B(\mathbf E^{ab})\rightarrow\B(\mathbf E)$ is an isomorphism. Hence Lemma \ref{4.7} implies that
\begin{equation}\label{4.9.1}
\textrm{\rm Sol}_{\mathbb A}^{\Bi}(\mathbf E)=
\textrm{\rm Sol}_{\mathbb A}(\mathbf e)^{-1}
\big(\textrm{\rm Sol}_{\mathbb A}^{\Bi}(\mathbf E^{ab})\big).
\end{equation}
First assume that $\textrm{\rm Sol}_{\mathbb A}^{\Bi}(\mathbf E)\neq\emptyset$. Obviously in this case $\textrm{\rm Sol}_{\mathbb A}(\mathbf E)\neq\emptyset$ and by (\ref{4.9.1}) we have:
$$\emptyset\neq\textrm{\rm Sol}_{\mathbb A}(\mathbf e)
\big(\textrm{\rm Sol}_{\mathbb A}^{\Bi}(\mathbf E)\big)\subseteq
\textrm{\rm Sol}_{\mathbb A}^{\Bi}(\mathbf E^{ab})$$
so Theorem \ref{1.1} implies that the set $\textrm{\rm Sol}(\mathbf E^{ab})$ is non-empty, too. Now assume that $\textrm{\rm Sol}_{\mathbb A}(\mathbf E)\neq\emptyset$ and $\textrm{\rm Sol}(\mathbf E^{ab})\neq\emptyset$. Then the homomorphism $\b_{\mathbf E^{ab}}$ is zero hence $\textrm{\rm Sol}_{\mathbb A}^{\Bi}(\mathbf E^{ab})=\textrm{\rm Sol}_{\mathbb A}(\mathbf E^{ab})$. So (\ref{4.9.1}) implies that
$$\textrm{\rm Sol}_{\mathbb A}^{\Bi}(\mathbf E)=
\textrm{\rm Sol}_{\mathbb A}(\mathbf e)^{-1}
\big(\textrm{\rm Sol}_{\mathbb A}(\mathbf E^{ab})\big)=
\textrm{\rm Sol}_{\mathbb A}(\mathbf E)\neq\emptyset\textrm{.}$$
\end{proof}
\begin{defn} Let $\overline G_1$ denote the quotient group $G_1/[\textrm{Ker}(\mathbf E),\textrm{Ker}(\mathbf E)]$ where we continue to use the notation above. For every $\pi\in\textrm{\rm Sol}(\mathbf E^{ab})$ let $\mathbf E_{\pi}$ be the embedding problem given by the diagram:
$$\CD @.\Gamma\\
@.@V\pi VV\\
G_1
@>e>>\overline G_1\endCD$$
Moreover let $r_{\pi}:\textrm{\rm Sol}_{\mathbb A}(\mathbf E_{\pi})\rightarrow
\textrm{\rm Sol}_{\mathbb A}(\mathbf E)$ denote the tautological inclusion. The following result is analogous to the statement (see \cite{CS}) that the descent obstruction for abelian covers is equivalent to the algebraic Brauer-Manin obstruction:
\end{defn}
\begin{thm}\label{4.10} Assume that $\textrm{\rm char}(F)$ does not divide the order of $\textrm{\rm Ker}(\mathbf E)_{ab}$. Then we have:
$$\bigcup_{\pi\in\textrm{\rm Sol}(\mathbf E^{ab})}
r_{\pi}\big(\textrm{\rm Sol}_{\mathbb A}(\mathbf E_{\pi})\big)=
\textrm{\rm Sol}_{\mathbb A}^{\textrm{\rm Br}_1}(\mathbf E).$$
\end{thm}
\begin{proof} The homomorphism $\textrm{\rm Br}(\mathbf e)|_{\textrm{Br}_1(\mathbf E^{ab})}:\textrm{Br}_1(\mathbf E^{ab})\rightarrow
\textrm{Br}_1(\mathbf E)$ is an isomorphism by Corollary \ref{4.6}, where $\mathbf e:\mathbf E\rightarrow\mathbf E^{ab}$ is the map introduced in the proof above. Therefore Lemma \ref{4.7} implies that
\begin{equation}\label{4.10.1}
\textrm{\rm Sol}_{\mathbb A}^{\textrm{Br}_1}(\mathbf E)=
\textrm{\rm Sol}_{\mathbb A}(\mathbf e)^{-1}
\big(\textrm{\rm Sol}_{\mathbb A}^{\textrm{Br}_1}(\mathbf E^{ab})\big)=
\textrm{\rm Sol}_{\mathbb A}(\mathbf e)^{-1}
\big(r(\textrm{\rm Sol}(\mathbf E^{ab}))\big),
\end{equation}
where in the second equation we used Theorem \ref{1.2}. For every $\pi\in\textrm{Sol}(\mathbf E^{ab})$ we have $\textrm{\rm Sol}_{\mathbb A}(\mathbf e)^{-1}
(r(\pi))=r_{\pi}\big(\textrm{\rm Sol}_{\mathbb A}(\mathbf E_{\pi})\big)$, and hence (\ref{4.10.1}) implies that
$$\textrm{\rm Sol}_{\mathbb A}^{\textrm{\rm Br}_1}(\mathbf E)=
\bigcup_{\pi\in\textrm{\rm Sol}(\mathbf E^{ab})}
r_{\pi}\big(\textrm{\rm Sol}_{\mathbb A}(\mathbf E_{\pi})\big)\textrm{.}$$
\end{proof}

\section{A local-global principle for cohomology classes}

\begin{notn} Let $M$ be a discrete finite abelian $\Gamma$-module whose order is not divisible by char$(F)$. For every $k\in\mathbb N$ and for every $x\in|F|$ where $M$ is unramified let $H^k_{un}(\Gamma_x,M)$ denote the image of the inflation map $H^k(\widehat{\mathbb Z},M)\rightarrow
H^k(\Gamma_x,M)$. Moreover for every $k$ as above let $H^k_{\mathbb A}(F,M)$ denote the subgroup:
$$H^k_{\mathbb A}(F,M)=\{\prod_{x\in|F|}\!\!c_x|\textrm{$c_x\in H^k_{un}(\Gamma_x,M)$ for almost all $x\in|F|$}\}\leq
\prod_{x\in|F|}H^k(\Gamma_x,M).$$
Let $\{\cdot,\cdot\}$ denote the duality pairing:
$$\CD H^1(\Gamma_x,M)\times H^1(\Gamma_x,M^{\vee})
@>\cup>>H^2(\Gamma_x,M\otimes M^{\vee})@>\textrm{ev}_*>>
H^2(\Gamma_x,\overline F_x^*)\endCD$$
of local class field theory where $\textrm{ev}_*$ is the map induced by the evaluation map $\textrm{ev}:M\otimes M^{\vee}
\rightarrow\overline F_x^*$.
\end{notn}
\begin{lemma} For every $\prod_{x\in|F|}c_x\in H^1_{\mathbb A}(F,M)$ and $\prod_{x\in|F|}d_x\in H^1_{\mathbb A}(F,M^{\vee})$ we have:
$$\{c_x,d_x\}=0$$
for almost all $x\in|F|$.
\end{lemma}
\begin{proof} The proof of the claim above is essentially the same as the proof of Lemma \ref{3.3}. We may assume without the loss of generality that at $x$ both $M$ and $M^{\vee}$ are unramified and the cohomology classes $c_x,d_x$ are in $H^1_{un}(\Gamma_x,M)$ and in $H^1_{un}(\Gamma_x,M^{\vee})$, respectively. In this case the cohomology class $\{c_x,d_d\}$ lies in the image of the inflation map $H^2(\widehat{\mathbb Z},(\overline F_x^*)^{I_x})\rightarrow H^2(\Gamma_x,\overline F_x^*)$ where $I_x\triangleleft\Gamma_x$ is the inertia subgroup. Because the group $H^2(\widehat{\mathbb Z},(\overline F_x^*)^{I_x})$ is trivial, the claim is now clear.
\end{proof}
\begin{notn} By the lemma above the pairing:
$$[\cdot,\cdot]:H^1_{\mathbb A}(F,M)\times H^1_{\mathbb A}(F,M^{\vee})\longrightarrow\mathbb Q/\mathbb Z$$
given by the rule
$$[\prod_{x\in|F|}c_x,\prod_{x\in|F|}d_x]=\sum_{x\in|F|}\textrm{\rm inv}_x(\{c_x,d_x\})$$
is well-defined because all but finitely many of the summands are zero. For every $k\in\mathbb N$ and for every $M$ as above let
$$\widetilde H^k(F,M)=\textrm{Im}\left(
\prod_{x\in|F|}\iota_x^*:H^k(\Gamma,M)\rightarrow
\prod_{x\in|F|}H^k(\Gamma_x,M)\right)\leq H^k_{\mathbb A}(F,M).$$
\end{notn}
\begin{thm}\label{5.4} Under the pairing the $[\cdot,\cdot]$ annulator of $\widetilde H^1(F,M^{\vee})$ is $\widetilde H^1(F,M)$.
\end{thm}
\begin{proof} Let $c\in H^1(\Gamma,M)$ and $d\in H^1(\Gamma,M^{\vee})$. for every $x\in|F|$ we have:
$$\{\iota_x^*( c ),\iota_x^*(d)\}=(\eta_x)_*\iota_x^*(\textrm{ev}_*(c\cup d))$$
where we let $\textrm{ev}_*$ also denote the map $H^2(\Gamma,M\otimes M^{\vee})\rightarrow H^2(\Gamma,\overline F^*)$ induced by the evaluation map $\textrm{ev}:M\otimes M^{\vee}\rightarrow\overline F^*$. Hence
$$[\prod_{x\in|F|}\iota_x^*(c),\prod_{x\in|F|}\iota_x^*(d)]=\sum_{x\in|F|}\textrm{inv}_x((\eta_x)_*\iota_x^*(\textrm{ev}_*(c\cup d)))=0$$
by the reciprocity law for Brauer groups over global fields applied to the cohomology class $\textrm{ev}_*(c\cup d)\in H^2(\Gamma,\overline F^*)$. So $\widetilde H^1(F,M)$ is contained by the annulator of $\widetilde H^1(F,M^{\vee})$ with respect to $[\cdot,\cdot]$. In order to continue our proof we need to introduce some notation.\let\qed\relax
\end{proof}
\begin{defn} Let $M$ be as above and let $S\subset|F|$ be a finite non-empty set that contains every archimedean place of $F$ and every place where $M$ is ramified. Let $\mathcal O_{F,S}\subset F$ denote the ring of $S$-integers in $F$. Because char$(F)$ does not divide the order of $M$ we may also assume that the latter is a unit in $\mathcal O_{F,S}$ by enlarging $S$, if it is necessary. For every $k\in\mathbb N$ and for every $M$ as above let
$$H^k_S(F,M)=\{c\in H^k(\Gamma,M)|
\iota_x^*(c)\in H^k_{un}(\Gamma_x,M)\ (\forall x\not\in S)\}\leq
H^k(\Gamma,M).$$
Let $\widetilde H^k_S(F,M)$ denote the image of $H^k_S(F,M)$ in $\prod_{x\in S}H^k(\Gamma_x,M)$ with respect to $\prod_{x\in S}\iota_x^*$. Finally let
$$[\cdot,\cdot]_S:\prod_{x\in S}H^k(\Gamma_x,M)\times\prod_{x\in S}H^k(\Gamma_x,M^{\vee})\longrightarrow\mathbb Q/\mathbb Z$$
denote the pairing given by the rule
$$[\prod_{x\in S}c_x,\prod_{x\in S}d_x]\mapsto\sum_{x\in S}\textrm{\rm inv}_x(\{c_x,d_x\}).$$
\end{defn}
\begin{lemma}\label{5.6} Under the pairing $[\cdot,\cdot]_S$ the annulator of $\widetilde H^1_S(\Gamma,M^{\vee})$ is $\widetilde H^1_S(\Gamma,M)$.
\end{lemma}
\begin{proof} By part $(c)$ of Theorem 4.10 of \cite{Mi} on pages 70-71 the sequence:
$$\CD H^1_S(F,M)@>\prod_{x\in S}\iota_x^*>>\prod_{x\in S}H^1(\Gamma_x,M)
@>{c\mapsto[c,\cdot]_S}>>H^1_S(F,M^{\vee})^*\endCD$$
is exact where the superscript $*$ denotes the dual Hom$(\cdot,\mathbb Q/\mathbb Z)$.
\end{proof}
Let $\Sh^1_S(F,M)$ denote the kernel of $\prod_{x\in S}\iota_x^*$ in $H^1_S(F,M)$.
\begin{lemma}\label{5.7} There is a finite subset $S\subset|F|$ of the type considered above such that the image of $\Sh^1_S(F,M)$ in $H^1_{\mathbb A}(F,M)$ with respect to $\prod_{x\in|F|}\iota_x^*$ is trivial.
\end{lemma}
\begin{proof} Let $R\subset|F|$ be a set of the type considered above. By part $(a)$ of Theorem 4.10 of \cite{Mi} on pages 70-71 the group $\Sh^1_R(F,M)$ is finite. Therefore there is a finite set $S\subset|F|$ containing $R$ such that
$$\bigcap_{x\in|F|}\textrm{Ker}(\iota_x^*|_{\Shi^1_R(F,M)})=
\bigcap_{x\in S}\textrm{Ker}(\iota_x^*|_{\Shi^1_R(F,M)}).$$
Because $\Sh^1_S(F,M)\leq\Sh^1_R(F,M)$ the claim is now clear.
\end{proof}
\begin{proof}[End of the proof of Theorem \ref{5.4}] By the above we only have to show that every $c=\prod_{x\in|F|}c_x\in H^1_{\mathbb A}(F,M)$ annihilated by $\widetilde H^1(F,M^{\vee})$ with respect to $[\cdot,\cdot]$ actually lies in $\widetilde H^1(F,M)$. By definition there is a finite subset $S\subset|F|$ of the type considered above such that $c_x\in H^1_{un}(\Gamma_x,M)$ for every $x\in|F|-S$. By Lemma \ref{5.7} we may also assume that the image of $\Sh^1_S(F,M)$ with respect to $\prod_{x\in|F|}\iota_x^*$ is trivial by enlarging $S$ if it is necessary. Let $R\subset|F|$ be finite subset containing $S$. Note that for every $d\in H^1_R(F,M^{\vee})$ we have:
$$[c,d]=[\prod_{x\in R}c_x,\prod_{x\in R}\iota^*_x(d)]_R$$
since for every $x\in|F|-R$ the cohomology classes $c_x$ and $\iota^*_x(d)$ are $H^1_{un}(\Gamma_x,M)$ and in $H^1_{un}(\Gamma_x,M^{\vee})$, respectively. Hence the element $\prod_{x\in R}c_x\in\prod_{x\in S}H^1(\Gamma_x,M)$ is annihilated by $\widetilde H^1_R(F,M^{\vee})$ with respect to the pairing $[\cdot,\cdot]_R$. Since $\mathcal O_{F,R}\supseteq\mathcal O_{F,S}$ the order of $M$ is invertible in $\mathcal O_{F,R}$ and hence there is a $c_R\in H^1_R(F,M)$ such that
\begin{equation}\label{5.7.1}
\iota^*_x(c_R)=c_x\quad(\forall x\in R)
\end{equation}
by Lemma \ref{5.6}. Now for any pair of finite subsets $R_1,R_2\subset|F|$ containing $S$ we have $c_{R_1}-c_{R_2}\in\Sh^1_S(F,M)$ by (\ref{5.7.1}). Therefore we get that for every $R\subset|F|$ the image of $c_R$ with respect to $\prod_{x\in|F|}\iota^*_x$ is independent of the choice of $R$ and it is actually equal to $c$, again by (\ref{5.7.1}).
\end{proof}

\section{A topological theorem on the cup product}

\begin{defn}\label{6.1} In this chapter all topological spaces are Hausdorff and locally contractible.  For every topological space $T$ and abelian group $A$ let $A_T$ denote the constant sheaf $A$ on $T$. Let $p:X\rightarrow Y$ be a fibre bundle with a connected fibre $F$. Let $r:Y\rightarrow X$ and $s:Y\rightarrow X$ be sections of the fibration $p$. Let $p_!$ denote the derived left adjoint of the pull-back functor $p^*$ from the category of complexes of sheaves on $Y$ to category of complexes of sheaves on $X$. (The adjoint $p_!$ exists because $p^*$ commutes with arbitrary limits, since we assumed $p$ to be a fibre bundle, and so we may apply Freyd's adjoint functor theorem.) By functoriality both $r$ and $s$ induces maps $r^h,s^h\in [\mathbb Z_Y,p_!(\mathbb Z_X)]$ in the derived category of complexes of sheaves on $Y$. Let $\deg\in[p_!(\mathbb Z_X),\mathbb Z_Y]$ denote the map from $p_!(\mathbb Z_X)$ onto its $0$-th homology. (The latter is $\mathbb Z_Y$ because we assumed that the fibres are connected.) Since both $r$ and $s$ are sections, the compositions of $r^h,s^h$ with $\deg$ are both the identity map in $[\mathbb Z_Y,\mathbb Z_Y]$. Let $\tau_{>0}(p_!(\mathbb Z_X))$ be the fibre of the map $\deg$ in the derived category and let
$$\CD\tau_{>0}(p_!(\mathbb Z_X))@>f_0>>p_!(\mathbb Z_X)@>{\deg}>>
\mathbb Z_Y@>>>\tau_{>0}(p_!(\mathbb Z_X))[1]\endCD$$
be the corresponding distinguished triangle. Therefore their difference $r^h-s^h\in [\mathbb Z_Y,p_!(\mathbb Z_X)]$ is the image of a map $[r-s]\in[\mathbb Z_Y,\tau_{>0}(p_!(\mathbb Z_X))]$ such that $r^h-s^h=f_0\circ[r-s]$. This map is unique since $[\mathbb Z_Y,\mathbb Z_Y[-1]]$ is zero. Let $\mathcal B$ be the first homology of the complex $\tau_{>0}(p_!(\mathbb Z_X))$, let $h_1:\tau_{>0}(p_!(\mathbb Z_X))\to\mathcal B[1]$ be the Postnikov truncation, and consider the distinguished triangle:
$$\CD\tau_{>1}(p_!(\mathbb Z_X))@>f_1>>\tau_{>0}(p_!(\mathbb Z_X))
@>{h_1}>>\mathcal B[1]@>>>\tau_{>1}(p_!(\mathbb Z_X))[1].\endCD$$
Let $\Delta(r,s)\in[\mathbb Z_Y,\mathcal B[1]]=H^1(Y,\mathcal B)$ denote $h_1\circ[r-s]$. 
\end{defn}
\begin{defn}\label{6.2} For every $y\in Y$ let $X_y=p^{-1}(y)$ be the fibre of $p$ over $y$ and let $i_y:X_y\rightarrow X$ denote the inclusion map. Let $\mathcal A$ be a locally constant sheaf of abelian groups on the base $Y$ and let $H^2(X,p^*(\mathcal A))_0$ denote the intersection of the kernels of the maps:
$$i_x^*:H^2(X,p^*(\mathcal A))\longrightarrow H^2(X_y,p^*(\mathcal A)|_{X_y})$$
for every $y\in Y$.  Note that
$$H^2(X,p^*(\mathcal A))_0=\bigcap_{y\in S}\textrm{ker}(i_y^*)$$
where $S\subseteq Y$ is any set such that for every connected component $C\subseteq Y$ there is a $y\in C\cap S$. Note that $H^2(X,p^*(\mathcal A))_0$ is the kernel of the edge homomorphism:
$$\epsilon:H^2(X,p^*(\mathcal A))\longrightarrow H^0(Y,R^2p_*(p^*(\mathcal A)))
$$
furnished by the Leray spectral sequence:
$$H^p(Y,R^qp_*(p^*(\mathcal A)))\Rightarrow
H^{p+q}(X,p^*(\mathcal A)).$$
So the higher edge homomorphism of this spectal sequence is a homomorphism:
$$\delta:H^2(X,p^*(\mathcal A))_0\longrightarrow H^1(Y,R^1p_*(p^*(\mathcal A)))
=H^1(Y,\textrm{Hom}(\mathcal B,\mathcal A))$$
where we used that there is a natural isomorphism:
$$R^1p_*(p^*(\mathcal A))=\textrm{Hom}(\mathcal B,\mathcal A).$$
Let
$$\cup:H^1(Y,\mathcal B)\times H^1(Y,\textrm{Hom}(\mathcal B,\mathcal A))\longrightarrow
H^2(Y,\mathcal A)$$
be the cup product furnished by the evaluation map:
$$\mathcal B\otimes\textrm{Hom}(\mathcal B,\mathcal A)\longrightarrow\mathcal A.$$
Note that every section $s$ of $p$ induces a homomorphism:
$$s^*:H^i(X,p^*(\mathcal A))\longrightarrow H^i(Y,s^*(p^*(\mathcal A)))=
H^i(Y,\mathcal A)\quad(\forall i\in\mathbb N).$$
\end{defn}
\begin{thm}\label{6.3} For every $c\in H^2(X,p^*(\mathcal A))_0$ and for every pair of sections $r,s$ of the fibration $p$ we have:
$$r^*(c)-s^*(c)=\Delta(r,s)\cup\delta(c)\in H^2(Y,\mathcal A).$$
\end{thm}
\begin{proof} For every topological space $T$ and complexes of sheaves $\mathcal C,\mathcal D$ and $\mathcal E$ on $T$ let
$$m(\mathcal C,\mathcal D,\mathcal E):[\mathcal C,\mathcal D]\times [\mathcal D,\mathcal E]\longrightarrow
[\mathcal C,\mathcal E]$$
denote the map given by the rule $(f,g)\mapsto f\circ g$. Moreover for every $f\in [\mathcal C,\mathcal D]$ let 
$$f\circ:[\mathcal D,\mathcal E]\longrightarrow
[\mathcal C,\mathcal E]$$
denote the map given by the rule $g\mapsto f\circ g$, and similarly for every 
$g\in [\mathcal D,\mathcal E]$ let 
$$\circ g:[\mathcal C,\mathcal D]\longrightarrow
[\mathcal C,\mathcal E]$$
denote the map given by the rule $f\mapsto f\circ g$. Then we have the following commutative diagram:
$$\!\!\xymatrix{
[\mathbb Z_Y,\mathcal B[1]] \times [\mathcal B[1],\mathcal A[2]]
\ar@<7ex>_{h_1\circ}[d]
\ar@<1ex>^{\ \ \ m(\mathbb Z_Y,\mathcal B[1],\mathcal A[2])}[rrrrrd] & & & & & \\
[\mathbb Z_Y,\tau_{>0}(p_!(\mathbb Z_X))]
\times [\tau_{>0}(p_!(\mathbb Z_X)),\mathcal A[2]]
\ar@<7ex>^{\circ h_1}[u]
\ar@<-7ex>_{\circ f_0}[d] 
\ar^{\ \ \ \ \ \ m(\mathbb Z_Y,\tau_{>0}(p_!(\mathbb Z_X)),\mathcal A[2])}[rrrrr]& & & & & [\mathbb Z_Y,\mathcal A[2]].
 \\
[\mathbb Z_Y,p_!(\mathbb Z_X)] \times [p_!(\mathbb Z_X),\mathcal A[2]] \ar@<-7ex>^{f_0\circ}[u] \ar@<-1ex>_{\ \ \ m(\mathbb Z_Y,p_!(\mathbb Z_X),\mathcal A[2])}[rrrrru]& & & & & }$$
For every section $t:Y\to X$ of $p$ let $t^h\in[\mathbb Z_Y,p_!(\mathbb Z_X)]$ be the map induced by $t$, similarly to the notation we introduced in Definition \ref{6.1}. Note that $[p_!(\mathbb Z_X),\mathcal A[2]] =H^2(X,\mathcal A)$ and under this identification for every section $t:Y\to X$ as above and $c\in H^2(X,\mathcal A)$ we have $t^h\circ c=t^*(c)$. Therefore
$$r^*(c)-s^*(c)=r^h\circ c-s^h\circ c=(r^h-s^h)\circ c=[r-s]\circ(f_0\circ c)$$
by the commutativity of the diagram above. Note that
$$[\mathcal B[1],\mathcal A[2]]=[\mathcal B,\mathcal A[1]]=H^1(Y,\textrm{Hom}(\mathcal B,\mathcal A)),$$ 
and under this identification $\delta(c)\in [\mathcal B[1],\mathcal A[2]]$ is such that $h_1\circ\delta(c)=f_0\circ c$. Indeed
$$[\tau_{>1}(p_!(\mathbb Z_X)),\mathcal A[2])]=
H^0(Y,R^2p_*(p^*(\mathcal A))$$
and under this identification
$$(f_1\circ f_0)\circ:[p_!(\mathbb Z_X),\mathcal A[2])]\longrightarrow
[\tau_{>1}(p_!(\mathbb Z_X)),\mathcal A[2])]$$
is the edge homomorphism $\epsilon$ in Definition \ref{6.2}. In particular the kernel of $(f_1\circ f_0)\circ$ is $H^2(X,\mathcal A)_0$. The second distinguished triangle in Definition \ref{6.1} induces a long exact sequence:
$$\CD [\tau_{>1}(p_!(\mathbb Z_X))[1],\mathcal A[2])]@>>>
[\mathcal B[1],\mathcal A[2]]@>{h_1\circ}>>
[\tau_{>0}(p_!(\mathbb Z_X)),\mathcal A[2])]\endCD$$
$$\CD@>{f_1\circ}>>
[\tau_{>1}(p_!(\mathbb Z_X)),\mathcal A[2])].\endCD$$
Since
$$[\tau_{>1}(p_!(\mathbb Z_X))[1],\mathcal A[2])]=
[\tau_{>1}(p_!(\mathbb Z_X)),\mathcal A[1])]=0,$$
there is a unique homomorphism:
$$\partial:H^2(X,\mathcal A)_0\longrightarrow[\mathcal B[1],\mathcal A[2]]=H^1(Y,\textrm{Hom}(\mathcal B,\mathcal A))$$
such that the diagram
$$\xymatrix{H^2(X,\mathcal A)_0
\ar[d]^{\partial}\ar[dr]^{f_0\circ|_{H^2(X,\mathcal A)_0}} &  \\
[\mathcal B[1],\mathcal A[2]]\ar[r]^{\!\!\!\!\!\!\!\!\!\!\!\!\!\!\!\!h_1\circ} &
[\tau_{>0}(p_!(\mathbb Z_X)),\mathcal A[2])]}$$
is commutative. The map $\partial$ is actually the higher edge homomorphism
$\delta$ in Definition \ref{6.2}. Now the relation $h_1\circ\delta(c)=f_0\circ c$ is clear. Now $[\mathbb Z_Y,\mathcal B[1]]=H^1(Y,\mathcal B)$ and $[\mathbb Z_Y,\mathcal B[2]]=H^2(Y,\mathcal B)$, and under these identifications $m(\mathbb Z_Y,\mathcal B[1],\mathcal A[2])$ is the cup product:
$$\cup:H^1(Y,\mathcal B)\times H^1(Y,\textrm{Hom}(\mathcal B,\mathcal A))\longrightarrow
H^2(Y,\mathcal A)$$
in Definition \ref{6.2} above. So by using the commutativity of the diagram above again we get that
$$[r-s]\circ(f_0\circ c)=\Delta(r,s)\circ\delta(c)=
\Delta(r,s)\cup\delta(c),$$
and the theorem follows.
\end{proof}
\begin{defn}\label{6.4} Let $M$ a finite abelian group and let
\begin{equation}\label{6.6.1}
\CD1@>>>M@>>>\Omega
@>>>\Pi@>>>1\endCD
\end{equation}
be an exact sequence in the category of prodiscrete groups. Let $A$ be a discrete $\Pi$-module and let $H^2(\Omega,A)_0$ denote the kernel of the restriction map:
$$H^2(\Omega,A)\longrightarrow H^2(M,A).$$
Moreover let
$$\delta:H^2(\Omega,A)_0\longrightarrow
H^1(\Pi,\textrm{Hom}(M,A))=
H^1(\Pi,H^1(M,A))$$
be the homomorphism furnished by the Hochschild-Serre spectral sequence:
$$H^p(\Pi,H^q(M,A))\Rightarrow
H^{p+q}(\Omega,A)$$
where we equip $M$ with its $\Pi$-module structure induced by the exact sequence (\ref{6.6.1}). For every pair of sections $s_1,s_2:\Pi\rightarrow\Omega$ of the exact sequence (\ref{6.6.1}) the $1$-cochain in $C^1(\Pi,M)$ given by the rule $g\mapsto s_1(g)s_2(g)^{-1}$ is actually is a cocycle. Let $[s_1-s_2]\in H^1(\Pi,M)$ be cohomology class represented by this cocycle. Finally let
$$\cup:H^1(\Pi,M)\times H^1(\Pi,\textrm{\rm Hom}(M,A))\longrightarrow
 H^2(\Pi,A)$$
 be the cup product induced by the evaluation map $M\otimes \textrm{Hom}(M,A)\rightarrow A$.
\end{defn}
\begin{cor}\label{6.5} For every $c\in H^2(\Omega,A)_0$ and for every pair of sections $s_1,s_2$ of the exact sequence {\rm (6.6.1)} we have:
$$s_1^*(c)-s_2^*(c)=[s_1-s_2]\cup\delta(c)\in H^2(\Pi,A).$$
\end{cor}
\begin{proof} We may assume that $\Pi$ is actually finite by applying the usual limit argument. The proof will be based on giving topological interpretation to both sides of the equation. The homomorphism $\Omega\to\Pi$ furnishes a Serre-fibration of classifying spaces $p:B\Omega\to B\Pi$ with fibre $BM$. The sections $s_1,s_2$ induce sections of the fibration $p$ which we will denote by the same symbols by abuse of notation. Let $\mathcal B$ denote the locally constant sheaf on $B\Pi$ corresponding to the $\Pi$-module $A$. Then $H^1(\Omega,M)=H^1(B\Omega,\mathcal B)$ and the cohomolology classes denoted by $[s_1-s_2]$ in Definitions \ref{6.1} and \ref{6.4} correspond to each other. There is a locally constant sheaf $\mathcal A$ on $B\Pi$ corresponding to the $\Pi$-module $A$. Note that $H^2(\Omega,A)=H^2(B\Omega,\mathcal A)$, and also $H^2(\Omega,A)_0=H^2(B\Omega,\mathcal A)_0$, where we use the notation of Definition \ref{6.2} for the fibration $p$. Moreover $H^1(\Omega,A)=H^2(B\Omega,\mathcal A)$, and the edge homomorphisms 
$$H^2(\Omega,A)_0\longrightarrow H^1(\Pi,\textrm{Hom}(M,A))
\textrm{ and } 
H^2(B\Omega,\mathcal A)_0\longrightarrow H^1(B\Pi,\textrm{Hom}(\mathcal B,\mathcal A))$$
correspond to each other under these identifications. Consequently the cohomology classes denoted by $\delta(c)$ in Definitions \ref{6.2} and \ref{6.4} also correspond to each other.  The claim now follows immediately from Theorem \ref{6.3}.
\end{proof}

\section{The Tate duality pairing and the Brauer-Manin pairing}

\begin{defn} By a continuous (or discrete) module over a pro-finite group
$\Delta$ we mean a $\Delta$-module $M$ such that the action of $\Delta$ is continuous with respect to the discrete topology on $M$. For every pro-finite group $\Delta$ let $\mathcal M(\Delta),\mathcal C(\Delta),\mathcal C^{+}(\Delta),\mathcal C^{-}(\Delta)$, and $\mathcal C^{\pm}(\Delta)$ denote the category of continuous $\Delta$-modules, the category of complexes of continuous $\Delta$-modules, the category of complexes in $\mathcal C(\Delta)$ bounded from above, the category of complexes in $\mathcal C(\Delta)$ bounded from below, and the category of complexes in $\mathcal C(\Delta)$ which are either bounded from above or below, respectively. For every object $C$ of $\mathcal C(\Delta)$ let $H_n(C)$ denote the $n$-th homology group of $C$. When $\Delta$ is the absolute Galois group $\Gamma=\textrm{Gal}(\overline F|F)$ of a field $F$, for every complex $C$
$$\CD\cdots@<<<C_{-1}@<<<C_0@<<<C_1@<<<\cdots\endCD$$
in $\mathcal C(\Delta)$ let $C^{\vee}$ denote the dual complex:
$$\xymatrix{\cdots&\textrm{Hom}(C_{1},\overline F^*)
\ar[l]&\textrm{Hom}(C_0,\overline F^*)\ar[l]&\textrm{Hom}(C_{-1},\overline F^*)\ar[l]&\cdots\ar[l]}$$
where Hom denotes the group of continuous group homomorphisms (and we equip $\overline F^*$ with the discrete topology). 
\end{defn}
\begin{defn} Note that for every pro-finite group $\Delta$ the category $\mathcal M(\Delta)$ has enough injectives, so right exact functors from $\mathcal M(\Delta)$ has derived functors. For every complex $C$ in $\mathcal C^{\pm}(\Delta)$ let $\HH^i(\Delta,C)$ denote its hypercohomology with respect to the functor of $\Delta$-invariants. Similarly for any object $C$ in
$\mathcal C^{\pm}(\Delta)$ let $\textrm{Ext}_{\Delta}^n(C,\cdot)$ denote the $n$-th derived functor of $\textrm{Hom}_{\mathcal C(\Delta)}(C,\cdot)$. When $\Delta=\Gamma=\textrm{Gal}(\overline F|F)$, as above, we will use the notation $\HH^i(F,C)$ for $\HH^i(\Gamma,C)$. When $F$ is a global field let
$$\HH^i_{\Pi}(F,C)=\prod_{x\in|F|}\HH^i(F_x,C)$$
where for every $x\in|F|$ we consider $C$ as an object of $\mathcal C(\Gamma_x)$ via the embedding $\iota_x:\Gamma_x\to\Gamma$, and we interpret $\HH^i(F_x,C)$ accordingly. For every such $x$ there is a pull-back map $i_x^*:\HH^i(F,C)\to\HH^i(F_x,C)$. Let 
$$\Sh^i(F,C)=\text{Ker}\big(\prod_{x\in|F|}\iota_x^*:
\HH^i(F, C)\to\HH^i_{\Pi}(F,C)\big).$$
\end{defn}
\begin{thm}\label{l:Sha} Let $F$ be a global field and let $C$ be a complex in $\mathcal C^{\pm}(\Gamma)$ such that $H_n(C)$ is finite for every $n$ and not divisible by the characteristic of $F$. Then there is a perfect pairing
$$\langle\cdot,\cdot\rangle:\Sh^i(F,C)\times\Sh^{3-i}(F,C^{\vee})\longrightarrow\mathbb Q/\mathbb Z.$$
\end{thm}
\begin{proof}
This is exactly Theorem 3.5.9 from ~\cite{Jos09}, when $F$ is a number field. The function field case can be proved exactly the same way. It is also important to note that this pairing specializes to the usual Poitou--Tate pairing for $\Gamma$-modules, i.e.~for complexes concentrated in degree zero.
\end{proof}
\begin{defn} Let $\mathbf E$ be an embedding problem over an arbitrary field $F$ given by the diagram (\ref{1.0.1}). Let $\mathbf E_*$ denote the contractible simplicial set freely generated by $G_1$. At the level of sets $\mathbf E_i=G_1^{i+1}$. The diagonal right-action of $G_1$ on each $\mathbf E_i$ induces a free right action of $G_1$ on $\mathbf E_*$, and therefore a free right action of Ker$(\mathbf E)$ on $\mathbf E_*$, too. Then we have a left action of $G_2$ on $\mathbf E_*/\textrm{Ker}(\mathbf E)$ and thus by pulling back with respect to $\phi$ a left action of $\Gamma$ on $\mathbf E_*/\textrm{Ker}(\mathbf E)$. Let $B(\mathbf E)_*$ denote this simplicial object in the category of $\Gamma$-sets. Let $\mathbb Z B(\mathbf E)_*$ denote the complex where $\mathbb Z B(\mathbf E)_n$ is the free abelian group generated by $B(\mathbf E)_n$ and the differential is the usual alternating sum. Equipped with the induced $\Gamma$-action this complex is an object of $\mathcal C^{+}(\Gamma)$. 
\end{defn}
\begin{defn} As a simplicial set $B(\mathbf E)_*$ is weakly equivalent to the Eilenberg--MacLane space $B\textrm{Ker}(\mathbf E)$, and hence
$$H_n(\mathbb Z B(\mathbf{E})_*)\cong
H_n(\textrm{Ker}(\mathbf E),\mathbb Z).$$
Let $\deg\in[\mathbb Z B(\mathbf{E})_*,\mathbb Z]$ denote the map from
$\mathbb Z B(\mathbf{E})_*$ onto its $0$-th homology, let
$\tau_{>0}(\mathbb Z B(\mathbf{E})_*)$ be the fibre of the map $\deg$ in the derived category, and let
$$\CD\tau_{>0}(\mathbb Z B(\mathbf{E})_*)@>f_0>>\mathbb Z B(\mathbf{E})_*@>{\deg}>>
\mathbb Z@>>>\tau_{>0}(\mathbb Z B(\mathbf{E})_*)[1]\endCD$$
be the corresponding distinguished triangle. By construction
$$H_n(\tau_{>0}(\mathbb Z B(\mathbf{E})_*))\cong 
\left\{\begin{array}{ll}
H_n(\textrm{Ker}(\mathbf E),\mathbb Z),&\text{if $n\neq0$,}
\\
0,&\text{if $n=0$.}\end{array}\right.$$
In particular when $\textrm{Ker}(\mathbf E)$ is abelian we have $H_1(\tau_{>0}(\mathbb Z B(\mathbf{E})_*))\cong H_1(\textrm{Ker}(\mathbf E),\mathbb Z)\cong\textrm{Ker}(\mathbf E)$. Let $h_1:\tau_{>0}(\mathbb Z B(\mathbf{E})_*)\to\textrm{Ker}(\mathbf E)[1]$ be the Postnikov truncation in this case.
 \end{defn}
\begin{defn}  Let
$$\pi^*_{\mathbf{E}}:\mathcal C(\Gamma)\to
\mathcal C(\Gamma(\mathbf E))$$
denote the functor which we get by pulling back with respect to the surjective homomorphism $\pi_{\mathbf E}:\Gamma(\mathbf E)\to\Gamma$. For any object $M$ of $\mathcal M(\Gamma(\mathbf E))$ let $\pi_{\mathbf{E}!}(M)$ denote the $\textrm{Ker}(\mathbf E)$-coinvariants of $M$, that is, the quotient of $M$ by the subgroup generated by the set:
$$\{x-\gamma(x)|x\in M,\gamma\in\textrm{Ker}(\mathbf E)\}.$$
Since the latter is a $\Gamma(\mathbf E)$-submodule, there is a natural action of $\Gamma$ on
$\pi_{\mathbf{E}!}(M)$, and hence we get a functor
$\pi_{\mathbf E}:\mathcal M(\Gamma(\mathbf E))\to
\mathcal M(\Gamma)$ which in turn induces a functor:
$$\pi_{\mathbf{E}!}:\mathcal C(\Gamma(\mathbf E))\longrightarrow
\mathcal C(\Gamma).$$
It can be easily seen that this functor is the left adjoint of
$\pi^*_{\mathbf E}$.
\end{defn}
\begin{defn} Let $\mathbb Z\mathbf E_*$ denote the chain complex of the contractible simplicial set $\mathbf E_*$. For every object $C$ of $\mathcal C(\Gamma(\mathbf E))$ we may take (the total complex of) the tensor product $C\otimes\mathbb Z\mathbf E_*$ in the category of complexes of $\mathbb Z$-modules and equip it with the diagonal $\Gamma(\mathbf E)$-action; this makes
$C\otimes\mathbb Z\mathbf E_*$ an object of $\mathcal C(\Gamma(\mathbf E))$. Let $\mathbb L\pi_{\mathbf{E}!}(C)$ denote $\pi_{\mathbf{E}!}(C\otimes\mathbb Z\mathbf E_*)$. As we will shortly see, the functor $\mathbb L\pi_{\mathbf{E}!}$ is the left derived functor of $\pi_{\mathbf{E}!}$ in a suitable interpretation, although the latter is not defined in the sense of classical homological algebra, as $\mathcal M(\Gamma(\mathbf E))$ does not have enough projectives. 
\end{defn}
\begin{lemma}\label{7.8a} There is an isomorphism:
$$\mathbb L\pi_{\mathbf{E}!}(\mathbb Z)\cong\mathbb Z B(\mathbf{E})_*.$$
\end{lemma}
\begin{proof} Clearly $\mathbb Z\otimes\mathbb Z\mathbf E_*
\cong\mathbb Z\mathbf E_*$ and $\pi_{\mathbf{E}!}(\mathbb Z\mathbf{E}_*)\cong\mathbb Z B(\mathbf{E})_*$.
\end{proof}
\begin{defn} Let $\Delta$ be any pro-finite group, as above, and for any pair $M,N$ of continuous $\Delta$-modules let
$\Hom_{\Delta}(M,N)$ denote the group of $\Delta$-module homomorphisms from $M$ to $N$. Now let $A=\{A_n\}_{n\in\mathbb Z},B=\{B_n\}_{n\in\mathbb Z}$ be two complexes in $\mathcal C(\Delta)$. Let $\underline{\Hom}_{\Delta}(A,B)=
\{\underline{\Hom}^n_{\Delta}(A,B)\}_{n\in\mathbb Z}$ be the equivariant mapping complex from $A$ to $B$, where
$$\underline{\Hom}^n_{\Delta}(A,B)=\prod_{i\in\mathbb Z}
\Hom_{\Delta}(A_i,B_{i-n}),$$
and the differential
$$d:\underline{\Hom}^n_{\Delta}(A,B)\lrar\underline{\Hom}^{n+1}_{\Delta}(A,B)$$
for any $f=\prod_{i\in\mathbb Z}f_i\in\prod_{i\in\mathbb Z}
\Hom_{\Delta}(A_i,B_{i-n})$ is given by
$$ d(f)_i=f_{i-1}\circ d_i^A+(-1)^nd_{i-n}^B\circ f_i,$$
where $d_i^A:A_i\to A_{i-1}$ and $d_i^B:B_i\to B_{i-1}$ are differentials of $A$ and $B$, respectively. We denote the kernel of $d$ by $\mathcal Z^n(A,B)\subseteq\underline{\Hom}^n_{\Delta}(A,B)$. Note that $\mathcal Z^n(A,B)$ consists of exactly those elements of $\underline{\Hom}^n_{\Delta}(A,B)$ which are maps of complexes of degree $n$ from $A$ to $B$.
\end{defn}
\begin{lemma}\label{adjointness} There are natural isomorphisms:
$$\Ext^n_{\Gamma}(\mathbb L\pi_{\mathbf E_!}(C),D)\cong
\Ext^n_{\Gamma(\mathbf E)}
(C,\pi^*_{\mathbf E}(D))
\quad(\forall n\in\mathbb N),$$
for every $C$ in $\mathcal C(\Gamma(\mathbf E))$ and $D$ in $\mathcal C(\Gamma)$.
\end{lemma}
\begin{proof} This isomorphism can be explained as an instance of Quillen adjunctions between model categories, or $\infty$-adjunctions between $(\infty,1)$-categories. However we will give
a simple direct proof. Let $\widetilde{D}$ be an resolution of $D$ by injective $\Gamma$-modules. The groups  $\Ext^n_{\Gamma}(\mathbb L\pi_{\mathbf E_!}(C),D)$ are the homologies of 
$\underline{\Hom}_{\Gamma}(\mathbb L\pi_{\mathbf E_!}(C), \widetilde{D})$. Then we have:
$$\underline{\Hom}_{\Gamma}(\mathbb L\pi_{\mathbf E_!}(C), \widetilde{D})\cong
\underline{\Hom}_{\Gamma}(\pi_{\mathbf{E}_!}(C\otimes
\mathbb Z\mathbf E_*),\widetilde{D})\cong
\underline{\Hom}_{\Gamma(\mathbf{E})}(C\otimes
\mathbb Z\mathbf E_*, \pi^*_{\mathbf{E}}(\widetilde{D})),$$
where we used the definition of $\mathbb L\pi_{\mathbf E_!}$ in the first isomorphism, and the fact that $\pi_{\mathbf{E}_!}$ is the left adjoint of $\pi^*_{\mathbf E}$ in the second. Moreover
$$\underline{\Hom}_{\Gamma(\mathbf{E})}(C\otimes
\mathbb Z\mathbf E_*, \pi^*_{\mathbf{E}}(\widetilde{D}))
\cong\underline{\Hom}_{\Gamma(\mathbf{E})}(C ,\underline{\underline{\Hom}}(\mathbb Z\mathbf E_*, \pi^*_{\mathbf{E}}(\widetilde{D}))),$$
where $\underline{\underline{\Hom}}(\cdot,\cdot)$ denotes the internal Hom in the category of $\Gamma(\mathbf E)$-complexes. (Explicitly  $\underline{\underline{\Hom}}(\cdot,\cdot)$ is the mapping complex of the underlying $\mathbb Z$-complexes which we equip with a continuous $\Gamma$-action via conjugation.) In order to conclude it is enough to note that
$\underline{\underline{\Hom}}(\mathbb Z\mathbf E_*, \pi^*_{\mathbf{E}}(\widetilde{D}))$ is an injective resolution of $\pi_{\mathbf{E}}^*(D)$.
\end{proof}
\begin{lemma}\label{7.10a} There are natural isomorphisms:
$$\HH^n(\Gamma(\mathbf E),\overline F^*)\cong
 \HH^n(\Gamma,\mathbb Z B(\mathbf{E})_*^{\vee})
 \quad(\forall n\in\mathbb N).$$
 \end{lemma}
 \begin{proof} By the uniqueness of $n$-th derived functors we have:
 $$\HH^n(\Gamma(\mathbf{E}),\overline F^*)
 \stackrel{\textrm{def}}{=}
 \HH^n(\Gamma(\mathbf{E}),\pi^*_\mathbf{E}(\overline F^*))
 \cong  \Ext^n_{\Gamma(\mathbf{E})}(\mathbb Z,\pi^*_\mathbf{E}(\overline F^*)),$$
as there is a natural isomorphism $\HH^n(\Delta,C)\cong  \Ext^n_{\Delta}(\mathbb Z,C)$ (where $C$ is an object of $\mathcal C(\Delta)$ and $\Delta$ is any pro-finite group). By Lemma \ref{adjointness} we have:
$$\Ext^n_{\Gamma(\mathbf{E})}(\mathbb Z,\pi^*_\mathbf{E}(\overline F^*))\cong\Ext^n_{\Gamma}(\mathbb L\pi_{\mathbf E_!}(\mathbb Z),\overline F^*).$$
Note that there is a spectral sequence:
$$\Ext^p_{\Gamma}(\mathbb Z,
\underline{\underline{\textrm{Ext}}}^q(\mathbb L\pi_{\mathbf E_!}(\mathbb Z),\overline F^*))\Rightarrow
\Ext^{p+q}_{\Gamma}(\mathbb L\pi_{\mathbf E_!}(\mathbb Z),\overline F^*),$$
where $\underline{\underline{\textrm{Ext}}}^*(\mathbb L\pi_{\mathbf E_!}(\mathbb Z),\cdot)$ is the derived functor of
$\underline{\underline{\Hom}}(\mathbb L\pi_{\mathbf E_!}(\mathbb Z),\cdot)$. Since $\overline F^*$ is divisible, this sequence degenerates, and hence we have an isomorphism:
$$\Ext^n_{\Gamma}(\mathbb L\pi_{\mathbf E_!}(\mathbb Z),\overline F^*)\cong
\Ext^n_{\Gamma}(\mathbb Z,\mathbb L\pi_{\mathbf E_!}(\mathbb Z)^{\vee}),$$
while by Lemma \ref{7.8a} and by the uniqueness of $n$-th derived functors we have:
$$\Ext^n_{\Gamma}(\mathbb Z,\mathbb L\pi_{\mathbf E_!}(\mathbb Z)^{\vee})\cong
\Ext^n_{\Gamma}(\mathbb Z,\mathbb ZB(\mathbf{E})_*^{\vee})
\cong
\HH^n(\Gamma,\mathbb ZB(\mathbf{E})_*^{\vee}).$$
\end{proof}
\begin{lemma}\label{7.10} Assume that $F$ is either a global or a local field. Then we have: 
$$\textrm{\rm Br}(\mathbf E)\cong\HH^2(\Gamma,\tau_{>0}(\mathbb Z B(\mathbf{E})_*)^{\vee}).$$
\end{lemma}
\begin{proof} The distinguished triangle:
$$\CD\tau_{>0}(\mathbb Z B(\mathbf{E})_*)@>f_0>>\mathbb Z B(\mathbf{E})_*@>{\deg}>>
\mathbb Z@>>>\tau_{>0}(\mathbb Z B(\mathbf{E})_*)[1]\endCD$$
gives rise to another distinguished triangle:
$$\CD\overline F^*
@>>>\mathbb Z B(\mathbf{E})_*^{\vee}@>>>
\tau_{>0}(\mathbb Z B(\mathbf{E})_*)^{\vee}
@>>>\overline F^*[1]\endCD$$
by taking duals. Since $H^3(\Gamma,\overline F^*)=0$ (see Proposition 15 of \cite{Se} on page 93 when $F$ is a local field, and see Corollary 4.21 of \cite{Mi}, page 80 when $F$ is a global field), the associated cohomological long exact sequence looks like:
$$\HH^2(\Gamma,\overline F^*) \to \HH^2(\Gamma,\mathbb Z B(\mathbf{E})_*^{\vee}) \to \HH^2(\Gamma,\tau_{>0}(\mathbb Z B(\mathbf{E})_*)^{\vee} )\to 0.$$
The first map is the composition:
$$\HH^2(\Gamma,\overline F^*)\cong 
\textrm{Ext}^2_{\Gamma}(\mathbb Z,\overline F^*)
\to\textrm{Ext}^2_{\Gamma}(\mathbb{Z} B(\mathbf{E})_*,\overline F^*)\cong \HH^2(\Gamma,\mathbb Z B(\mathbf{E})_*^{\vee}),$$
where the middle map is induced by the degree map deg$:\mathbb Z B(\mathbf{E})_* \to \mathbb{Z}$. The derived adjunction $\mathbb L\pi_{\mathbf E_!} \dashv \pi^*$ give rise to a co-unit map: $\mathbb L\pi_{\mathbf E_!}(\pi^*(\mathbb{Z})) \to \mathbb{Z}$ which, under the identification in Lemma \ref{7.8a}, is  deg. Using Lemma \ref{7.10a} the first map of the sequence above can be viewed as a homomorphism:
$$H^2(\Gamma,\overline F^*)=
\HH^2(\Gamma,\overline F^*) \to 
\HH^2(\Gamma,\mathbb Z B(\mathbf{E})_*^{\vee})=
H^2(\Gamma(\mathbf{E}),\overline F^*).$$
As this map is induced by the co-unit, it is the pull-back map (with respect to the surjection $\Gamma(\mathbf E)\to\Gamma$). The cokernel of the latter is Br$(\mathbf E)$ by definition, so the claim follows.
\end{proof}
\begin{cor}\label{7.13} Assume that $F$ is a global field. Then we have: 
$$\B(\mathbf{E})\cong\Sh^2(F,\tau_{>0}(\mathbb Z B(\mathbf{E})_*)^{\vee}).$$
\end{cor}
\begin{proof} This follows from Lemma \ref{7.10} applied to $F$ and all its completions. 
\end{proof}
\begin{rem}\label{7.14} The Postnikov truncation
$$h_1:\tau_{>0}(\mathbb Z B(\mathbf{E})_*)\to
\textrm{Ker}(\mathbf E)[1]$$
furnishes an isomorphism:
$$\Sh^1(F,\tau_{>0}(\mathbb Z B(\mathbf E)_*)
\longrightarrow
\Sh^1(F,\textrm{Ker}(\mathbf E)[1])\cong
\Sh^2(F,\textrm{Ker}(\mathbf E)).$$
\end{rem}
\begin{defn}\label{7.15a} Let $s:\Gamma\to\Gamma(\mathbf E)$ be a continuous section
of the map $\pi_{\mathbf E}:\Gamma(\mathbf E)\to\Gamma$. The identity map
$\textrm{id}\in\textrm{Ext}^0_{\Gamma}(\mathbb L\pi_{\mathbf E_!}(\mathbb Z),\mathbb L\pi_{\mathbf E_!}(\mathbb Z))$ via the isomorphism:
$$\textrm{Ext}^0_{\Gamma}(\mathbb L\pi_{\mathbf E_!}(\mathbb Z),\mathbb L\pi_{\mathbf E_!}(\mathbb Z)))\cong
\textrm{Ext}^0_{\Gamma(\mathbf E)}
(\mathbb Z,\pi^*_{\mathbf E}(\mathbb L\pi_{\mathbf E_!}(\mathbb Z)))$$
furnished by Lemma \ref{adjointness} furnishes an element $\textrm{id}_{\mathbf E}\in
\textrm{Ext}^0_{\Gamma(\mathbf E)}
(\mathbb Z,\pi^*_{\mathbf E}(\mathbb L\pi_{\mathbf E_!}(\mathbb Z)))$. 
By pulling back with respect to $s$ we get an element:
$$s^*(\textrm{id}_{\mathbf E})\in
\textrm{Ext}^0_{\Gamma}(s^*(\mathbb Z),s^*(\pi^*_{\mathbf E}(
\mathbb L\pi_{\mathbf E!}(\mathbb Z)))).$$
Since $s^*\circ\pi^*_{\mathbf E}=\textrm{id}_{\Gamma}^*=\textrm{id}$ and
$s^*(\mathbb Z)\cong\mathbb Z$, we get an element:
$$[s]\in
\textrm{Ext}^0_{\Gamma}(\mathbb Z,\mathbb L\pi_{\mathbf E!}(\mathbb Z))\cong \HH^0(\Gamma, \mathbb Z B(\mathbf E)_*)$$
(using Lemma \ref{7.8a}), which we will call the classifying element of the section $s$. Note that the map
$$\HH^0(\Gamma,\mathbb Z B(\mathbf E )_*)\longrightarrow
\HH^0(\Gamma,\mathbb Z)\cong\mathbb Z$$
induced by $\deg$ sends $[s]$ to $1$.
\end{defn}
Note that each element in $b\in\HH^2(\Gam(\mathbf{E}),\overline F^*)$ can be considered as an element of $\HH^2(\Gamma,\mathbb ZB(\mathbf{E})_*^{\vee})$ via the isomorphism in Lemma \ref{7.10a}. Let
$$\cup: \HH^2(\Gamma,\mathbb ZB(\mathbf{E})_*^{\vee}) \times \HH^0(\Gamma,\mathbb ZB(\mathbf{E})_*) \to
\HH^2(\Gamma ,\overline F^*)$$
be the cup product induced by the natural bilinear pairing:
$$\mathbb ZB(\mathbf{E})_*^{\vee}\times
\mathbb ZB(\mathbf{E})_*\longrightarrow
\overline F^*$$
of complexes.
\begin{lemma}\label{7.15} We have:
$$b\cup[s] = s^*(b)\in \HH^2(\Gamma,\overline F^*)$$
for every $b\in\HH^2(\Gam(\mathbf{E}),\overline F^*)$ and continuous section $s$ of $\pi_{\mathbf E}$.
\end{lemma}
\begin{proof} This claim follows at once from comparing the cup product above with the Yoneda pairing:
$$\textrm{Ext}^0_{\Gamma}(\mathbb Z,\mathbb ZB(\mathbf{E})_*)
\times\textrm{Ext}^2_{\Gamma}(\mathbb ZB(\mathbf{E})_*,\overline F^*) \to\textrm{Ext}^2_{\Gamma}(\mathbb Z,\overline F^*)$$
via the isomorphisms 
$$\HH^0(\Gamma,\mathbb ZB(\mathbf{E})_*)\cong \textrm{Ext}^0_{\Gamma}(\mathbb Z,\mathbb ZB(\mathbf{E})_*),\quad
\HH^2(\Gamma,\mathbb ZB(\mathbf{E})_*^{\vee})\cong
\textrm{Ext}^2_{\Gamma}(\mathbb ZB(\mathbf{E})_*,
\overline F^*),$$
$$\HH^2(\Gamma,\overline F^*)\cong \textrm{Ext}^0_{\Gamma}(\mathbb Z,\overline F^*).$$ 
We leave the details to the reader. 
\end{proof}
\begin{defn}\label{7.16} Let $M$ be a discrete finite abelian $\Gamma$-module. For every embedding problem $\mathbf E$ over $F$ such that $\textrm{Ker}(\mathbf E)=M$ and the set $\textrm{Sol}_{\mathbb A}(\mathbf E)$ is non-empty let $c_{\mathbf E}\in H^2(F,M)=H^2(F,\textrm{Ker}(\mathbf E))$ denote the class of the extension (\ref{1.3.1}). Note that $c_{\mathbf E
}\in\Sh^2(F,M)$ since we assumed that $\textrm{Sol}_{\mathbb A}(\mathbf E)$ is non-empty. Conversely for every $c\in \Sh^2(F,M)$ there is an embedding problem $\mathbf E$ as above such that $c=c_{\mathbf E}$. Let
\begin{equation}
\b:\Sh^1(F,M^{\vee})\times\Sh^2(F,M)\rightarrow\mathbb Q/\mathbb Z
\end{equation}
be the unique pairing such that $\b(b,c_{\mathbf E})=\b_{\mathbf E}(b)$ for every $b\in \Sh^1(F,M^{\vee})$ and embedding problem $\mathbf E$ as above. Since for every $b\in \Sh^1(F,M^{\vee})$ the value of $\b_{\mathbf E}(b)$ only depends on the isomorphism class of the embedding problem $\mathbf E$ (see the proof of Lemma \ref{4.7}), the pairing $\b$ is well-defined. Assume now that $\textrm{\rm char}(F)$ does not divide the order of $M$ and let
$$\tau:\Sh^1(F,M^{\vee})\times\Sh^2(F,M)\longrightarrow
\mathbb Q/\mathbb Z$$
denote the Tate duality pairing.
\end{defn}
\begin{thm}\label{pairing_theorem} We have $\b=-\tau$.
\end{thm}
We think that this theorem is very interesting on its own, since it gives an elegant description of the Tate duality pairing. It will proved in the rest of  this section. We will continue to use the notation which we have introduces so far. Let $\mathbf E$ an embedding problem of the type considered in Definition \ref{7.16}. Consider the cohomological long exact sequence:
$$\CD\HH^0(F,\mathbb ZB(\mathbf{E})_*)@>{\deg}>>\HH^0(F,\mathbb Z)
@>\partial>>\HH^1(F,\tau_{>0}(\mathbb ZB(\mathbf{E})_*))\endCD$$
corresponding to the distinguished triangle:
$$\CD\tau_{>0}(\mathbb Z B(\mathbf{E})_*)@>f_0>>\mathbb Z B(\mathbf{E})_*@>{\deg}>>
\mathbb Z@>>>\tau_{>0}(\mathbb Z B(\mathbf{E})_*)[1],
\endCD$$
and set $\delta=\partial(1)\in\Sh^1(F,\tau_{>0}(\mathbb Z B(\mathbf{E})_*))$. 
\begin{lemma}\label{7.19} The image of the classifying element $c_{\mathbf E}\in\Sh^2(F,\textrm{\rm Ker}(\mathbf E))$ under the isomorphism
$$\Sh^2(F,\textrm{\rm Ker}(\mathbf E))\cong\Sh^1(F,\tau_{>0}(\mathbb Z B(\mathbf{E})_*))$$
in Remark \ref{7.14} is the $\delta$ above.
\end{lemma}
\begin{proof} This is just a direct computation involving the representing cocycles. The details are left to the reader.
\end{proof}
Let
$$\langle\cdot,\cdot\rangle:\Sh^1(F,\tau_{>0}(\mathbb Z B(\mathbf{E})_*)) \times \Sh^2(F,\tau_{>0}(\mathbb Z B(\mathbf{E})_*)^{\vee})\longrightarrow
\mathbb Q/\mathbb Z $$
be the perfect pairing in Theorem~\ref{l:Sha}. Now let $b\in \Sh^2(F,\tau_{>0}(\mathbb Z B(\mathbf{E})_*)^{\vee})$ be arbitrary. Since $\HH^3(F,\overline F^*)=0$, the map
$$\HH^2(F,\mathbb Z B(\mathbf{E})_*^{\vee})\to \HH^2(F,\tau_{>0}(\mathbb Z B(\mathbf{E})_*)^{\vee})$$
is onto, and thus $b$ can be lifted to an element $\overline{b}\in \HH^2(F,\mathbb Z B(\mathbf{E})_*^{\vee})$, and we can take the localisation map to obtain
$$\overline{b}_x\stackrel{\textrm{def}}{=}\iota^*_x(\overline{b})\in
\HH^2(F_x,\mathbb Z B(\mathbf{E})_*^{\vee}).$$
Recall that we assumed that $\textrm{Sol}_{\mathbb A}(\mathbf E)$ is non-empty, so for every $x\in |F|$ let $h_x$ be a solution of $\mathbf E_x$ such that $h_x$ is unramified for almost all $x$. Let $s_x$ denote the section $s(h_x)$ corresponding to $h_x$ for every $x\in|F|$.  Finally let $\cup$ denote the cup product introduced after Definition \ref{7.15a} (over any field, including all completions of $F$).
\begin{prop}\label{7.20} We have the equality:
$$\langle\delta,b\rangle=-\sum_{x\in|F|}\inv_x([s_x]\cup\overline{b}_x)\in\mathbb Q/\mathbb Z.$$
\end{prop}
By Lemma \ref{7.15} the right hand side is $-\b_{\mathbf E}(b)$. On the other hand the isomorphisms:
$$\Sh^1(F,\textrm{\rm Ker}(\mathbf E)^{\vee})\cong\Sh^2(F,\tau_{>0}(\mathbb Z B(\mathbf{E})_*)^{\vee})
,
\Sh^2(F,\textrm{\rm Ker}(\mathbf E))\cong\Sh^1(F,\tau_{>0}(\mathbb Z B(\mathbf{E})_*))$$
in Corollary \ref{7.13} and Remark \ref{7.14} respect the pairing in the sense that the resulting diagram:
$$\xymatrix{
\Sh^1(F,\textrm{\rm Ker}(\mathbf E)^{\vee})
\ar@<-8ex>[d]\times \Sh^2(F,\textrm{\rm Ker}(\mathbf E))\ar@<9ex>[d]
  \ar[r]^-{\langle\cdot,\cdot\rangle} & \mathbb Q/\mathbb Z
\ar@{=}[d]  \\
\Sh^2(F,\tau_{>0}(\mathbb Z B(\mathbf{E})_*^{\vee})
 \times \Sh^1(F,\tau_{>0}(\mathbb Z B(\mathbf{E})_*))
\ar[r]^-{\langle\cdot,\cdot\rangle}  & \mathbb Q/\mathbb Z}$$
is commutative. Therefore by Lemma \ref{7.19} the left hand side is $\tau(c_{\mathbf E},b)$. So Theorem \ref{pairing_theorem} follows from Proposition \ref{7.20}, and hence we only have to prove the latter.
\begin{defn} For every pro-finite group $\Delta$ and open normal subgroup $U\leq\Delta$ let $E(\Delta/U)$ denote the standard (bar) resolution complex by free $\mathbb{Z}[\Delta/U]$-modules of $\mathbb{Z}$, equipped with the tautological $\Delta$-action. Note that since $E(\Delta/U)$ is quasi-isomorphic to $\mathbb{Z}$, any map of degree $i$ between $E(\Delta/U)$ and another complex $C$ in $\mathcal C(\Delta)$ gives rise to a hypercohomology class in $\mathbb{H}^i(\Delta,C)$. For every $g\in\mathcal Z^i(E(\Delta/U),C)$ let $[g]$ denote the class in $\mathbb{H}^i(\Delta,C)$ represented by $g$. 
\end{defn}
\begin{notn} Now let $\Delta'$ be another pro-finite group, let $U'\leq\Delta'$ be an open normal subgroup, and let $\phi:\Delta'\to\Delta$ be a continuous homomorphism such that $\phi(U')\subseteq U$. Then $\phi$ induces a homomorphism $\Delta'/U'\to\Delta/U$, which induces a map $E(\Delta'/U')\to E(\Delta/U)$ of complexes, which furnishes a homomorphism
$$\phi^*:\underline{\Hom}^i_{\Delta}(E(\Delta/U),C)\to\underline{\Hom}^i_{\Delta'}(E(\Delta'/U'),C)$$
compatible with the pull-back map on cohomology. We will drop
$\phi^*$ from the notation when $\phi$ is the identity map on
$\Delta$.
\end{notn}
\begin{defn} Now let $A=\{A_n\}_{n\in\mathbb Z},B=\{B_n\}_{n\in\mathbb Z}$ and $C=\{C_n\}_{n\in\mathbb Z}$ be three complexes in $\mathcal C(\Delta)$ such that there is a pairing:
$$m:A\otimes B\longrightarrow C.$$
Let $U$ and $E(\Delta/U)$ be as above, and let 
$$c:E(\Delta/U)\longrightarrow E(\Delta/U)\otimes E(\Delta/U)$$
denote the Alexander--Whitney map (see formula (1.4) of \cite{Bro} on page 108). Now consider the composition:
$$\xymatrix@C=2cm{
\underline{\Hom}^*_{\Delta}(E(\Delta/U),A)\times
\underline{\Hom}^*_{\Delta}(E(\Delta/U),B)\ar[d] \\
\underline{\Hom}^*_{\Delta}(E(\Delta/U)\otimes E(\Delta/U),A
\otimes B)\ar[d] \\
\underline{\Hom}^*_{\Delta}(E(\Delta/U),A
\otimes B)\ar[d] \\
\underline{\Hom}^*_{\Delta}(E(\Delta/U),C),}$$
where the first map is furnished by the functorial property of tensor products, the second is induced by the co-multiplication $c$, and the third map is induced by the multiplication $m$. Let $\cup$ denote the resulting pairing of complexes: 
$$\underline{\Hom}^*_{\Delta}(E(\Delta/U),A)\times
\underline{\Hom}^*_{\Delta}(E(\Delta/U),B)\longrightarrow
\underline{\Hom}^*_{\Delta}(E(\Delta/U),C).$$
Note the induced map on the cohomology is the exterior cup product. 
\end{defn}
\begin{proof}[Proof of Proposition \ref{7.20}] In order to prove the statement we shall use an explicit description of the pairing
$$\langle\cdot,\cdot\rangle:\Sh^1(F,\tau_{>0}(\mathbb Z B(\mathbf{E})_*))\times
\Sh^2(F,\tau_{>0}(\mathbb Z B(\mathbf{E})_*^{\vee})
\longrightarrow\mathbb Q/\mathbb Z$$
similar to the one given by Milne in \S I.4 of ~\cite{Mi}. Let $c\in\mathbb Z B(\mathbf{E})_0$ be such that $\deg(c)=1$. Denote by $U\triangleleft\Gamma$ the stabiliser of $c$. Let $g\in 
\underline{\Hom}^0_{\Gamma}(E(\Gamma/U),\mathbb Z B(\mathbf{E})_*)$ be such that $g_0(\sig) = \sig c$ for $\sig \in \Gamma/U$ and $g_i = 0$ for $i \neq 0$. Note that
$[\deg\circ g]$ represents $1 \in\HH^0(F,\mathbb Z)$ and so
$ \alpha =dg\in\mathcal Z^1(E(\Gamma/U),\tau_{>0}(\mathbb Z B(\mathbf{E})_*))$ represents $\delta=\partial(1)\in\HH^1(F,\tau_{>0}(\mathbb Z B(\mathbf{E})_*))$. Shrink $U$ enough so that one can represent $b$ by a map $\beta\in \mathcal Z^2\left(E(\Gamma/U),\tau_{>0}(\mathbb Z B(\mathbf{E})_*)^{\vee}\right)$, and $\overline b$ by
$\overline\beta\in\mathcal Z^2(E(\Gamma/U),
\mathbb Z B(\mathbf{E})_*^{\vee})$.
Now set
$$ \epsilon = g \cup\overline\beta\in
\underline{\Hom}^2_{\Gamma}(E(\Gamma/U),\overline F^*).$$
Note that $d\epsilon=dg\cup\overline\beta=\alpha\cup\overline\beta$. Set $g_x=\iota_x^*(g)$ for every $x\in|F|$. For every place $x$ we can take a small enough $U_x\triangleleft\Gamma_x$ such that we can represent $g_x$ in $\underline{\Hom}^0_{\Gamma}
(E(\Gamma_x/U_x),\mathbb Z B(\mathbf{E})_*)$ (that is, we have $U_x\subseteq\Gamma_x\cap U$), and we can represent $[s_x]\in\HH^0(F_x,\mathbb Z B(\mathbf{E})_*)$ by an $f_x\in \mathcal Z^0(E(\Gamma_x/U_x),\mathbb Z B(\mathbf{E})_*)$. Denote $h_x = g_x-f_x$. Then
$$dh_x = dg_x-df_x = \alpha_x $$
where $\alpha_x=\iota_x^*(\alpha)$ (for every $x\in|F|$). Since $\deg(h_x)=0$ we see that $h_x$ actually lies in
$\underline{\Hom}^0_{\Gamma}
(E(\Gamma_x/U_x),\tau_{>0}(\mathbb Z B(\mathbf{E})_*))$. Hence we can cup it with $\beta_x=\iota_x^*(\beta)\in\mathcal Z^1(E(\Gamma_x/U_x),\tau_{>0}(\mathbb Z B(\mathbf{E})_*)^{\vee})$ and get an element in
$\underline{\Hom}^1_{\Gamma}(E(\Gamma_x/U_x),
\overline F^*_x)$. We then observe that
$$ d(h_x\cup\beta_x) = dh_x \cup\beta_x= \alp_x \cup\beta_x
= d\epsilon_x,$$
where $\epsilon_x=\iota_x^*(\epsilon)$ (for every $x\in|F|$), and so we can define
$$ c_x = [h_x \cup\beta_x-\epsilon_x] \in \HH^2(F_x,
\overline F^*_x).$$
Our generalisation for Milne's formula is the following expression for the pairing:
$$\langle\delta,b\rangle= \sum_{x\in|F|}\inv_x(c_x)\in
\mathbb Q/\mathbb Z.$$
Now by naturality
$$h_x \cup\beta_x=h_x \cup\overline\beta_x,$$
where the first cup is computed in $\tau_{>0}(\mathbb Z B(\mathbf{E})_*),\tau_{>0}(\mathbb Z B(\mathbf{E})_*)^{\vee}$ and the second in $\mathbb Z B(\mathbf{E})_*,\mathbb Z B(\mathbf{E})_*^{\vee}$. We then get
$$ c_x = [h_x \cup\beta_x-\epsilon_x] = [h_x\cup\beta_x-
g_x\cup\overline\beta_x] = [h_x \cup\overline\beta_x-g_x\cup\overline\beta_x] = [-f_x \cup\overline\beta_x] = -[s_x]
\cup\overline b_x,$$
because $\overline b_x=[\overline\beta_x]$ (for every $x\in|F|$).
\end{proof}

\section{Proof of the main results}

\begin{proof}[Proof of Theorem \ref{1.1}] The implications:
$$\textrm{\rm Sol}(\mathbf E)\neq0\Rightarrow
\textrm{\rm Sol}_{\mathbb A}^{\textrm{\rm Br}}(\mathbf E)\neq0
\Rightarrow
\textrm{\rm Sol}_{\mathbb A}^{\textrm{\rm Br}_1}(\mathbf E)\neq0
\Rightarrow
\textrm{\rm Sol}_{\mathbb A}^{\Bi}(\mathbf E)\neq0$$
are trivially true. Assume now that $\textrm{\rm Sol}_{\mathbb A}^{\Bi}(\mathbf E)\neq0$. Then the homomorphism:
$$\b_{\mathbf E}:\Sh^1(\textrm{Ker}(\mathbf E)^{\vee})
\rightarrow\mathbb Z/\mathbb Q$$
is zero by definition. Hence the cohomology class $c_{\mathbf E}\in
H^2(\Gamma,\textrm{Ker}(\mathbf E))$ of the extension (\ref{1.3.1}) is annihilated by the pairing $\b$ of (\ref{1.3.2}). By part $(a)$ of Theorem 4.10 of \cite{Mi} on page 70 the pairing $\tau$ is perfect. Therefore the pairing $\b$ is also perfect by Theorem \ref{1.4}. The claim is now clear.
\end{proof}
\begin{proof}[Proof of Theorem \ref{1.2}] Because of the inclusions:
$$r(\textrm{Sol}(\mathbf E))\subseteq\textrm{Sol}_{\mathbb A}^{\textrm{Br}}(\mathbf E)\subseteq\textrm{Sol}_{\mathbb A}^{\textrm{Br}_1}(\mathbf E)$$
we only have to show that $\textrm{Sol}_{\mathbb A}^{\textrm{Br}_1}(\mathbf E)
\subseteq r(\textrm{Sol}(\mathbf E))$. This claim is trivial when $\textrm{Sol}_{\mathbb A}^{\textrm{Br}_1}(\mathbf E)$ is empty. Assume now that
$\textrm{Sol}_{\mathbb A}^{\textrm{Br}_1}(\mathbf E)\neq\emptyset$. Then by Theorem \ref{1.1} the set $\textrm{Sol}(\mathbf E)$ is non-empty. Therefore the short exact sequence (\ref{1.3.1}) splits. Choose such a splitting $h\in\textrm{Sol}(\mathbf E)$; such a choice furnishes a natural bijection:
$$\alpha:\textrm{Sol}(\mathbf E)\longrightarrow H^1(\Gamma,\textrm{Ker}(\mathbf E))$$
between the sections of the short exact sequence (\ref{1.3.1}) and the cohomology group $H^1(\Gamma,\textrm{Ker}(\mathbf E))$ mapping $h$ to zero. Moreover the splitting above induces a splitting of the short exact sequence:
\begin{equation}\label{8.0.1}
\CD1@>>>\textrm{Ker}(\mathbf E)@>i_{\mathbf E_x}>>\Gamma_x(\mathbf E_x)@>\pi_{\mathbf E_x}>>\Gamma_x@>>>1\endCD
\end{equation}
for every $x\in|F|$ which in turn furnishes a natural bijection:
$$\alpha_x:\textrm{Sol}(\mathbf E_x)\longrightarrow H^1(\Gamma_x,\textrm{Ker}(\mathbf E))$$
mapping $r_x(h)$ to zero. Moreover for every $x\in|F|$ where $\textrm{Ker}(\mathbf E)$ is unramified we have $\alpha_x(\textrm{Sol}_{un}(\mathbf E_x))=H^1_{un}(\Gamma_x,\textrm{Ker}(\mathbf E))$ hence there is a commutative diagram
\begin{equation}\label{8.0.2}
\CD\textrm{Sol}(\mathbf E)
@>\alpha>>H^1(\Gamma,\textrm{Ker}(\mathbf E))\\
@V{r}VV@VV\prod_{x\in|F|}\iota_x^*V\\
\textrm{Sol}_{\mathbb A}(\mathbf E)
@>\alpha_{\mathbb A}>>
H^1_{\mathbb A}(\Gamma,\textrm{Ker}(\mathbf E))\endCD
\end{equation}
where
$$\alpha_{\mathbb A}=\prod_{x\in|F|}\alpha_x|_{\textrm{Sol}_{\mathbb A}(\mathbf E)}.$$
Because the vertical maps in the diagram (\ref{8.0.2}) are bijections the claim follows from Theorem \ref{5.4} and the proposition below.
\end{proof}
\begin{prop} For every $g\in\textrm{\rm Sol}_{\mathbb A}(\mathbf E)$ and for every $c\in H^1(\Gamma,\textrm{\rm Ker}(\mathbf E)^{\vee})$ we have:
$$\langle g,j_{\mathbf E}(c)\rangle=[\alpha_{\mathbb A}(g),
\prod_{x\in|F|}\iota^*_x(c)].$$
\end{prop}
\begin{proof} Let $g=\prod_{x\in|F|}g_x$. Then we have:
\begin{align}
\langle g,j_{\mathbf E}(c)\rangle=&\sum_{x\in|F|}\textrm{\rm inv}_x(s(g_x)^*(j_{\mathbf E_x}^*(\iota^*_x(c))))\nonumber\\
=&\sum_{x\in|F|}\textrm{\rm inv}_x(s(g_x)^*(j_{\mathbf E_x}^*(\iota^*_x(c)))-s(h_x)^*(j_{\mathbf E_x}^*(\iota^*_x(c))))
+\langle h,j_{\mathbf E}(c)\rangle\nonumber\\
=&\sum_{x\in|F|}\textrm{\rm inv}_x(\{\alpha_x(g_x),\iota^*_x(c)\})+\langle h,j_{\mathbf E}(c)\rangle\nonumber\\
=&\ [\alpha_{\mathbb A}(g),\prod_{x\in|F|}\iota^*_x(c)]\nonumber
\end{align}
where the third equation follows from Corollary \ref{6.5} applied to the short exact sequence (\ref{8.0.1}) for every $x\in|F|$ and the fourth equation is a consequence of Lemma \ref{3.6}.
\end{proof}

\section{A geometric construction}

\begin{notn}\label{9.1} For every scheme $V$ defined over a field $F$ let $\overline V$ denote the base change of $V$ to $\overline F$. In this paper by the cohomology of a smooth group scheme $G$ over a base scheme $V$ we mean the cohomology of the sheaf it represents for the \'etale topology on $V$. For every geometrically connected variety $V$ over a field $F$ let $\pi_1(V)$ denote the isomorphism class of the \'etale fundamental group of $V$ with respect to any geometric point as a base point. 
\end{notn}
\begin{defn}\label{9.2} Let $V$ be a geometrically connected variety defined over $F$. Let $\eta$ be a $\overline F$-valued point of $V$. Then Grothendieck's short exact sequence of \'etale fundamental groups for $V$ is:
\begin{equation}\label{fundamentalsequence}
\CD1@>>>\pi_1(\overline V,\eta)@>>>
\pi_1(V,\eta)@>>>\Gamma@>>>1,\endCD
\end{equation}
which is an exact sequence of profinite groups in the category of topological groups. Every $F$-rational point $x\in V(F)$ induces a section $\Gamma\rightarrow\pi_1(V,\eta)$ of the sequence (\ref{fundamentalsequence}), well-defined up to conjugation by $\pi_1(\overline V,\eta)$. Let $\textrm{Sec}(V/F)$ denote the set of conjugacy classes of sections of (\ref{fundamentalsequence}) (in the category of profinite groups where morphisms are continuous homomorphisms). Then we have a map:
$$s_{V/F}:V(F)\longrightarrow\textrm{Sec}(V/F)$$
which sends every point $x\in V(F)$ to the corresponding conjugacy class of sections. This map is called the {\it section map}. Note that for a different base point $\eta'\in X(\overline F)$ there is a canonical identification between the corresponding section maps, so it is justified to suppress them from the notation.
\end{defn}
\begin{defn}\label{9.3} Let $\mathbf E$ be an embedding problem over the field $F$ given by the diagram (\ref{1.0.1}). Equip  $G_2$ with the trivial $\Gamma$-action. Then 
$$H^1(F,G_2)=\textrm{Hom}(\Gamma,G_2)/\sim,$$
where $\sim$ is the conjugacy relation. Let $[\psi]\in H^1(F,G_2)$ be the class 
corresponding to the homorphism $\psi\in\textrm{Hom}(\Gamma,G_2)$ defining $\mathbf E$. Corresponding to the cohomology class $[\psi]$ there is a right $G_2$-torsor over $F$ which we will denote by $T(\mathbf E)$. Consider $G_1$ as a subgroup of $G(\mathbf E)=SL_{|G_1|+1}$ over $F$
via its augmented regular representation (the regular representation with an additional dimension to fix the determinant). Let $X(\mathbf E)$ be quotient of 
$G(\mathbf E)\times T(\mathbf E)$ via the diagonal action of $G_1$ on the right, where we let $G_1$ act on $T(\mathbf E)$ via the homomorphism $\phi$. We will call $X(\mathbf E)$ the {\it classifying space} of the embedding problem $\mathbf E$. The left action of $G(\mathbf E)$ on the first factor of the product $G(\mathbf E)\times T(\mathbf E)$ descends uniquely to $X(\mathbf E)$ making the quotient map $G(\mathbf E)\times T(\mathbf E)\to X(\mathbf E)$ into a $G(\mathbf E)$-equivariant morphism.  
\end{defn}
\begin{lemma}\label{9.4} The following holds:
\begin{enumerate}
\item[$(i)$] with respect to the action defined above $X(\mathbf E)$ is a homogeneous space over $G(\mathbf E)$, and the geometric stabiliser of $X(\mathbf E)$ is $\textrm{\rm Ker}(\mathbf E)$,
\item[$(ii)$] $\pi_1(\overline{X(\mathbf E)})\cong\textrm{\rm Ker}(\mathbf E)$,
\item[$(iii)$] Grothendieck's short exact sequence of \'etale fundamental groups for $V=X(\mathbf E)$ is:
$$\CD1@>>>\textrm{\rm Ker}(\mathbf E)@>>>
\Gamma(\mathbf E)@>>>\Gamma@>>>1,\endCD$$
the short exact sequence {\rm (\ref{1.3.1})} associated to the embedding problem $\mathbf E$.
\end{enumerate}
\end{lemma}
\begin{proof} Since the homomorphism $\phi$ is surjective, the base change
$\overline{X(\mathbf E)}$ is connected, so the action of $\overline{G(\mathbf E)}$ is transitive, and hence $X(\mathbf E)$ is a homogeneous space over $G(\mathbf E)$. Moreover the geometric stabiliser of $X(\mathbf E)$ is the kernel of $\phi$, so the first claim is clear. Since $G(\mathbf E)$ is simply connected, claim $(ii)$ follows from $(i)$. Now fix an $\overline F$-valued point $\eta$ of $X(\mathbf E)$ and let $\pi_2:\pi_1(X(\mathbf E),\eta)\to
\Gamma$ be the surjection supplied by the short exact sequence (\ref{fundamentalsequence}). Moreover let $\pi_1:\pi_1(X(\mathbf E),\eta)\to
G_1$ be the homomorphism corresponding to the Galois cover $G(\mathbf E)\times T(\mathbf E)\to X(\mathbf E)$. Then $\pi_1\times\pi_2$ is injective by part $(ii)$. Since the composition $\phi\circ\pi_1$ is $\psi\circ\pi_2$ by construction, we get that the image of $\pi_1(X(\mathbf E),\eta)$ with respect to $\pi_1\times\pi_2$ lies in $\Gamma(\mathbf E)$. As the projection
$\pi_2$ is surjective, and its kernel is isomorphic to $\textrm{\rm Ker}(\mathbf E)$, we get that the image of $\pi_1\times\pi_2$ is exactly $\Gamma(\mathbf E)$. Claim $(iii)$ is now clear.
\end{proof}
\begin{defn}\label{9.5} By part $(iii)$ of above $\textrm{Sec}(X(\mathbf E)/F)=\textrm{Sol}(\mathbf E)$. Since $\overline{G(\mathbf E)}$ is simply connected, the section map is constant on the orbits of the action of $G(\mathbf E)$, so $s_{X(\mathbf E)/F}$ furnishes a map:
$$\sigma_{\mathbf E}:X(\mathbf E)(F)/G(\mathbf E)(F)\longrightarrow{\textrm{\rm Sol}}(\mathbf E).$$ 
\end{defn}
\begin{thm}\label{homo_space} The map $\sigma_{\mathbf E}$ is a bijection.
\end{thm}
\begin{proof} We may assume that ${\textrm{\rm Sol}}(\mathbf E) \neq \emptyset$ without the loss of generality. Let $s$ be an arbitrary section of Grothendieck's short exact sequence associated to $X(\mathbf E)$. The image Im$(s)$ of $s$ is a closed subgroup of $\pi_1(X(\mathbf E),\eta)$ with finite index, and hence it is an open subgroup, too. Let $\pi_s:X_s\to X(\mathbf E)$ be the finite, \'etale cover corresponding to this open subgroup, that is, for some choice of an $\overline F$-valued base point $\gamma$ of $X_s$ mapping to $\eta$ the image of the homomorphism $\pi_1(\pi_s):\pi_1(X_s,\gamma)\to\pi_1(X(\mathbf E),\eta)$ induced by $\pi_s$ is Im$(s)$. Note that the base change of this cover to $\overline F$ is the universal cover of $\overline{X(\mathbf E)}$, so in particular $\overline X_s$ is isomorphic to $\overline{G(\mathbf E)}$. Consider the following commutative diagram:
$$\xymatrix@C=2cm{
& & X_s\ar^{\pi_s}[d] \\
G(\mathbf E)\times X_s \ar_{\textrm{id}_{G(\mathbf E)}\times\pi_s}[r] 
\ar@{.>}^-{m_s}[urr]
&
G(\mathbf E)\times X(\mathbf E)\ar_{m_{\mathbf E}}[r] & X(\mathbf E),}$$
where $m_{\mathbf E}$ is the action of $G(\mathbf E)$ on $X(\mathbf E)$. Since the product $G(\mathbf E)\times X_s$ is geometrically simply connected, the pull-back of the cover $\pi_s$ with respect to the composition $m_{\mathbf E}\circ(\textrm{id}_{G(\mathbf E)}\times\pi_s)$ is constant. The restriction of the pull-back onto $X_s\cong\{1\}\times X_s$, where $1\in G(\mathbf E)$ is the identity element, is the pull-back of $\pi_s$ with respect to $\pi_s$. Therefore it has a section, namely the diagonal map. Since the pull-back onto $G(\mathbf E)\times X_s$ is constant, this diagonal map has a unique extension to a section on $G(\mathbf E)\times X_s$, which supplies a
map $m_s:G(\mathbf E)\times X_s\to X_s$ making the diagram above commutative and whose restriction onto $X_s\cong\{1\}\times X_s$ is the identity map. 

So in short, we get that the cover $X_s$ is equipped with a right $G(\mathbf E)$-action such that $\pi_s$ is $G(\mathbf E)$-equivariant. Since the base change of this action to $\overline F$ is the usual right action of $\overline{G(\mathbf E)}$ on itself $\overline{G(\mathbf E)}=\overline{X}_s$, we get that $X_s$ is a $G(\mathbf E)$-torsor. Consider the generalised $\Gamma$-equivariant short exact sequence:
$$\CD 1@>>> SL_{|G_1|+1}(\overline F)@>>> GL_{|G_1|+1}(\overline F)
@>>>\overline F^*@>>>1\endCD$$
of pointed sets equipped with a continuous $\Gamma$-action. (Here continuous means that the stabiliser of every point is open in $\Gamma$.) The corresponding cohomological long exact sequence of pointed sets is:
$$\cdots\to GL_{|G_1|+1}(F)\to F^*\to H^1(F,SL_{|G_1|+1}(\overline F))
\to  H^1(F,GL_{|G_1|+1}(\overline F))\to\cdots$$
By Hilbert's theorem 90 for the general linear group the last term is zero. Since the determinant is surjective, we get that the term $H^1(F,SL_{|G_2|+1}(\overline F))=0$ is also zero. Therefore every $G(\mathbf E)$-torsor is trivial, and so $X_s$ is isomorphic to $G(\mathbf E)$. 

Note that $s$ lies in $s_{X(\mathbf E)/F}(x)$ for a rational point $x\in X(\mathbf E)$ if and only if $x$ has an $F$-rational lift to $X_s$ with respect to $\pi_s$. By the above $X_s(F)$ is non-empty for every section $s$, so the map $\sigma_{\mathbf E}$ is surjective. Now let $x,y\in X(\mathbf E)$ be such that $s_{X(\mathbf E)/F}(x)=s_{X(\mathbf E)/F}(y)$. Choose a representative $s$ of this common conjugacy class. As we already noted there are $x',y'\in X_s(F)$ mapping under $\pi_s$ to $x$,$y$, respectively. Since $X_s$ is the trivial $G(\mathbf E)$-torsor, the points $x'$ and $y'$ are in the same $G(\mathbf E)$-orbit. Since $\pi_s$ is $G(\mathbf E)$-equivariant, the same is true for $x$ and $y$. We get that $\sigma_{\mathbf E}$ is injective, too.
\end{proof}
\begin{defn}\label{9.7} Let $V$ be again a geometrically connected smooth $F$-variety. We may calculate the \'etale cohomology groups $H^i(V,\mathbb G_m)$ by considering all hypercoverings of $V$. One can restrict to only those that come from connected Galois coverings of $V$ and thus get natural maps:
$$\rho_V^i:H^i(\pi_1(V),\overline F^*)\longrightarrow
H^i(V,\mathbb G_m),$$
where the action of the pro-finite group $\pi_1(V)$ on $\overline F^*$ is furnished via the projection $\pi_1(V)\rightarrow\Gamma$.
\end{defn}
\begin{lemma}\label{9.8} The map:
$$\rho^i_{\overline{X(\mathbf E)}}:
H^i(\pi_1(\overline{X(\mathbf E)}),\overline F^*)\longrightarrow H^i(\overline{X(\mathbf E)},\mathbb G_m)$$
is an isomorphism for $i=0,1,2$.
\end{lemma}
\begin{proof} By part $(ii)$ of Lemma \ref{9.4} we have
$\pi_1(\overline{X(\mathbf E)})=\textrm{\rm Ker}(\mathbf E)$, and by part $(i)$ we have $\overline{X(\mathbf E)}=\overline{G(\mathbf E)}/\textrm{\rm Ker}(\mathbf E)$. Consider the Hochschild-Serre spectral sequence:
$$H^p(\textrm{\rm Ker}(\mathbf E),H^q(\overline{G(\mathbf E)},\mathbb G_m))
\Rightarrow H^{p+q}(\overline{G(\mathbf E)}/\textrm{\rm Ker}(\mathbf E),\mathbb G_m).$$
Now $H^0(\overline{G(\mathbf E)},\mathbb G_m)=\overline F^*$, and the edge homomorphism:
$$H^i(\textrm{\rm Ker}(\mathbf E),\overline F^*)=
H^i(\textrm{\rm Ker}(\mathbf E),H^0(\overline{G(\mathbf E)},\mathbb G_m))
\longrightarrow H^i(\overline{X(\mathbf E)},\mathbb G_m),$$
of this spectral sequence is $\rho^i_{\overline{X(\mathbf E)}}$. Since $G(\mathbf E)=SL_{|G_1|+1}$, we get that
$$H^1(\overline{G(\mathbf E)},\mathbb G_m)=H^2(\overline{G(\mathbf E)},\mathbb G_m)=0$$
(for example because the \'etale homotopy type of $\overline{G(\mathbf E)}$ is $2$-connected). The claim is now clear.
\end{proof}
\begin{lemma}\label{9.9} The natural map
$$\rho^i_{X(\mathbf E)}:H^i(\pi_1(X(\mathbf E)),\overline F^*)\longrightarrow H^i(X(\mathbf E),\mathbb G_m)$$
is an isomorphism for $i=0,1, 2$.
\end{lemma}
\begin{proof} The short exact sequence
$$\CD 1@>>> \textrm{\rm Ker}(\mathbf E)@>>>
\pi_1(X(\mathbf E)) @>>>\Gamma @>>>1\endCD$$
funishes a Hochschild--Serre spectral sequence:
$$H^p(F,H^q(\textrm{\rm Ker}(\mathbf E),\overline F^*))
\Rightarrow H^{p+q}(\pi_1(X(\mathbf E)),\overline F^*).$$
There is also another Hochschild--Serre spectral sequence:
$$H^p(F,H^q(\overline{X(\mathbf E)},\mathbb G_m))
\Rightarrow H^{p+q}(\overline{G(\mathbf E)}/\textrm{\rm Ker}(\mathbf E),\overline F^*).$$
There is a map from the first spectral sequence to the second (induced by the map of \'etale topoi) such that the corresponding homomorphisms:
$$H^p(\pi_1(X(\mathbf E)),H^q(\textrm{\rm Ker}(\mathbf E),\overline F^*))\longrightarrow H^{p,q}(X(\mathbf E),\mathbb G_m)$$
of the $E_2^{p,q}$ terms are induced by the $\Gamma$-module maps
$\rho^q_{\overline{X(\mathbf E)}}$. Since the latter are isomorphisms for $q<3$ by Lemma \ref{9.8}, the claim follows.
\end{proof}
\begin{notn}\label{9.10} Recall that for any scheme $V$ over a field $F$ the relative Brauer group Br$(V/F)$ is by definition the cokernel of the map
$$\omega_V:H^2(\textrm{Spec}(F),\mathbb G_m)\to H^2(V,\mathbb G_m)$$ induced by the structural morphism $V\to\textrm{Spec}(F)$. When $V=X(\mathbf E)$ then the diagram:
$$\xymatrix{
H^2(\pi_1(X(\mathbf E)),\overline F^*)\ar^{\rho^2_{X(\mathbf E)}}[rr] & & H^2(X(\mathbf E),\mathbb G_m) \\
& H^2(\Gamma,\overline F^*),\ar^{\pi^*_{\mathbf E}}[ul]
\ar_{\omega_{X(\mathbf E)}}[ur] & }$$
is commutative, and hence the map $\rho^2_{X(\mathbf E)}$ induces an isomorphism
$$\textrm{Br}(\mathbf E)\longrightarrow\textrm{Br}(X(\mathbf E)/F)$$
by Lemma \ref{9.9}, which we will denote by $\beta_{\mathbf E}$.
\end{notn}
\begin{defn}\label{9.11} Assume now that $F$ is a number field. Let $V$ be a scheme over a field $F$, and for every $x\in|F|$ let $V_x$ denote the base change of $V$ to $F_x$. Let $\B(V/F)\leq\textrm{Br}(V/F)$ denote the set of all elements whose image under the map $\textrm{Br}(V/F)\to\textrm{Br}(V_x/F_v)$ induced by the base change map $H^2(V,\mathbb G_m)\to H^2(V_x,\mathbb G_m)$ is zero.
\end{defn}
\begin{lemma}\label{9.12} The restriction of $\beta_{\mathbf E}$ onto $\B(\mathbf E)$ induces an isomorphism:
$$\beta_{\mathbf E}|_{\Bi(\mathbf E)}:\B(\mathbf E)\longrightarrow
\B(X(\mathbf E)/F).$$
\end{lemma}
\begin{proof} Note that the base change of $X(\mathbf E)$ to $F_x$ is $X(\mathbf E_x)$ for every $x\in|X|$. The claim now immediately follows from the fact that the maps $\beta_{\mathbf E}$ and $\beta_{\mathbf E_x}$ (for every $x\in|X|$) are all isomorphisms. 
\end{proof}
\begin{defn}\label{9.13b} Let $V$ be again a geometrically connected smooth $F$-variety. Let
$$\langle\cdot,\cdot\rangle:V(K)\times H^i(V,\mathbb G_m)
\longrightarrow \textrm H^i(\textrm{Spec}(K),\mathbb G_m)=
H^i(\Gamma,\overline F^*)$$
denote the pairing which is the pull-back of the class in with respect to the section
$\textrm{Spec}(K)\to V$ corresponding to the $K$-rational point. Note that for every section $s:\Gamma\to\pi_1(V,\eta)$ of the short exact sequence (\ref{fundamentalsequence}) and every $c\in H^i(\pi_1(V),\overline F^*)$ the pull-back $s^*(c)\in H^i(\Gamma,\overline F^*)$ only depends on the conjugacy class of $s$. Therefore we have a well-defined pairing:
$$\{\cdot,\cdot\}:\textrm{\rm Sec}(V/F)
\times H^i(\pi_1(V),\overline F^*)
\longrightarrow \textrm H^i(\Gamma,\overline F^*)$$
defined by taking the pull-back of the right argument with respect to some representative of the left argument.  
\end{defn}
\begin{lemma}\label{9.14a} The diagram:
$$\xymatrix{
V(K)\ar@<-7ex>[d]_{s_{V/K}}\times H^i(V,\mathbb G_m)
  \ar[r]^-{\langle\cdot,\cdot\rangle} & H^i(\Gamma,\overline F^*) \ar@{=}[d]  \\
\textrm{\rm Sec}(V/F)
 \times H^i(\pi_1(V),\overline F^*)
\ar@<-5ex>[u]^{\rho_V^i}\ar[r]^-{\{\cdot,\cdot\}}  & H^i(\Gamma,\overline F^*),}$$
commutes.
\end{lemma}
\begin{proof} The lemma is immediate from the naturality of the map $\rho_V^i$.
\end{proof}
\begin{defn}\label{9.13} We will continue to work under the assumptions of Definition \ref{9.11}, and to use the notation there. Assume that $V_x(F_x)\neq\emptyset$ for every $x\in|F|$. Let $b\in\B(V/F)$ be arbitrary and choose a $c\in H^2(V,\mathbb G_m)$ which maps to $b$ under the quotient map. For every $x\in|F|$ let $c_x\in H^2(V_x,\mathbb G_m)$ be the base change of $c$ to $V_x$. Note that for every $\prod_{x\in|F|}p_x\in\prod_{v\in|F|}V_x(F_x)$ the sum:
$$\sum_{x\in|F|}\textrm{\rm inv}_x(p_x^*(c_x))\in\mathbb Q/\mathbb Z$$
is finite, (i.e.~all but finitely many summands are zero), and its value $\mathbf i(b)$ in $\mathbb Q/\mathbb Z$ only depends on $b$, not on the choice of $c$ or the $p_x$. (See 6.2 of \cite{Sa} why.)  We say that the unramified Brauer--Manin obstruction is the only one for the local--global principle for a class of varieties over $F$ if for every $V$ in the class such that $V_x(F_x)\neq\emptyset$ for every $x\in|F|$ the condition:
$$\mathbf i(b)=0\quad(\forall b\in\B(V/F))$$
implies that $V(F)$ is non-empty.
\end{defn}
We will need the following crucial theorem due to Borovoi:
\begin{thm}\label{borovoi} The unramified Brauer--Manin obstruction is the only one for the local--global principle for homogeneous spaces over geometrically simply connected reductive groups with abelian geometric stabiliser.
\end{thm}
\begin{proof} See Theorem 2.2 of \cite{Bo1} on page 185.
\end{proof}
Now we are ready to give a geometric
\begin{proof}[Proof of Theorem \ref{1.1}] Assume now that $F$ is a number field. As we already noted we only need to show that $\textrm{\rm Sol}_{\mathbb A}^{\Bi}(\mathbf E)\neq0$ implies that $\textrm{\rm Sol}(\mathbf E)\neq0$. For the sake of simple notation let $X$ denote $X(\mathbf E)$ and for every $v\in|F|$ let $X_v$ denote the base change of $X$ to $F_v$. Note that $X_v=X(\mathbf E_v)$ by construction. For every $v\in|F|$ there is a bijection:
$$\sigma_{\mathbf E_v}:X_v(F_v)/G(\mathbf E_x)(F_v)\longrightarrow{\textrm{\rm Sol}}(\mathbf E_v),$$ 
by Theorem \ref{homo_space}, so our assumption implies that for every $v\in|F|$ the set $X_v(F_v)$ is non-empty. Note that $\mathbf i(\beta_{\mathbf E}(b))=\b_{\mathbf E}(b)$ for every $b\in\B(\mathbf E)$ by Lemmas \ref{9.4} and \ref{9.14a}, so the claim follows from Theorem \ref{borovoi} combined with Lemma \ref{9.12}.
\end{proof}
In the rest of this section we want to apply our geometric construction to the problem of exhibiting splitting varieties for Massey products. Such splitting varieties were constructed by Hopkins and Wickelgren in \cite{HW} for order-$3$ Massey products. Our more general construction will give splitting varieties for Massey products of any order. 
\begin{defn}\label{11.1} Let $\mathcal C^*$ be a differential graded associative algebra with product $\cup$, differential $\delta:\mathcal C^*\to\mathcal C^{*+1}$, and cohomology $H^*=\textrm{Ker}(\delta)/\textrm{Im}(\delta)$. Choose an integer $n\geq2$ and let $a_1,a_2,\ldots,a_n$ be a set of cohomology classes in $H^1$. A defining system for the order-$n$ Massey product of $a_1,a_2,\ldots,a_n$ is a set $a_{ij}$ of elements of $\mathcal C^1$ for $1\leq i<j\leq n+1$ and $(i,j)\neq(1,n+1)$ such that
$$\delta(a_{ij})=\sum_{k=i+1}^{j-1}a_{ik}\cup a_{kj}$$
and $a_1,a_2,\ldots,a_n$ is represented by $a_{12},a_{23},\ldots,a_{n,n+1}$. We say that the order-$n$ Massey product of $a_1,a_2,\ldots,a_n$ is defined if there exists a defining system. The order-$n$ Massey product $\langle a_1,a_2,\ldots,a_n\rangle_{a_{ij}}$ of $a_1,a_2,\ldots,a_n$ with respect to the defining system $a_{ij}$ is the cohomology class of
$$\sum_{k=2}^na_{1k}\cup a_{k,n+1}$$
in $H^2$. Let $\langle a_1,a_2,\ldots,a_n\rangle$ denote the subset of $H^2$ consisting of the order-$n$ Massey products of $a_1,a_2,\ldots,a_n$ with respect to all defining systems. 
\end{defn}
\begin{defn}\label{11.2} Let $e_{ij}:\textrm{Mat}_n(\mathbb Z/2\mathbb Z)\to\mathbb Z/2\mathbb Z$ be the function taking an
$n\times n$ matrix with coefficients in
$\mathbb Z/2\mathbb Z$ to its $(i, j)$-entry. Let
$$\mathcal U_n=\{U\in\textrm{Mat}_n(\mathbb Z/2\mathbb Z)\ |\ e_{ii}(U)=1,\ e_{ij}(U)=0\  (\forall\ i>j)\}$$
be the group of unipotent $n\times n$ matrices with coefficients in
$\mathbb Z/2\mathbb Z$. Let $\Gamma$ again denote the absolute Galois group of a field $F$, and let $\mathcal C^*$ be now the differential graded algebra of $\mathbb Z/2\mathbb Z$-cochains in continuous group cohomology. Then $H^1$ is naturally isomorphic to Hom$(\Gamma,\mathbb Z/2\mathbb Z)$. Given $n$ homomorphisms
$$a_i:\Gamma\longrightarrow\mathbb Z/2\mathbb Z\quad(i=1,2,\ldots,n),$$
let $\mathbf E(a_1,a_2,\ldots,a_n)$ denote the embedding problem:
$$\CD @.\Gamma\\
@.@VV{a_1\times a_2\times\cdots\times a_n}V\\
\mathcal U_{n+1}@>\phi>>(\mathbb Z/2\mathbb Z)^{n},\endCD$$
where $\phi$ is given by the rule $U\mapsto(e_{12}(U),e_{23}(U),\ldots,e_{nn+1}(U))$.
\end{defn}
We will need the following result due to Dwyer:
\begin{thm}\label{11.3} The order-$n$ Massey product $\langle a_1,a_2,\ldots,a_n\rangle$ is defined and contains $0$ if and only if the embedding problem $\mathbf E(a_1,a_2,\ldots,a_n)$ has a solution.
\end{thm}
\begin{proof} See Theorem 2.4 of \cite{Dw}.
\end{proof}
\begin{defn}\label{11.4} Let
$$\kappa:F^*\longrightarrow H^1(\Gamma,\mathbb Z/2\mathbb Z)=
\textrm{Hom}(\Gamma,\mathbb Z/2\mathbb Z)$$
be the Kummer map associated to the short exact sequence:
$$\CD 0@>>>\mathbb Z/2\mathbb Z@>>>\overline F^*@>{x\mapsto x^2}>> 
\overline F^* @>>> 0\endCD$$
of $\Gamma$-modules. Let $a_1,a_2,\ldots,a_n\in F^*$ be arbitrary. We say that a variety $X$ over $F$ is a {\it splitting variety} for the order-$n$ Massey product $\langle\kappa(a_1),\kappa(a_2),\ldots,\kappa(a_n)\rangle$ if $X$ has an $F$-rational point if and only if $\langle\kappa(a_1),\kappa(a_2),\ldots,\kappa(a_n)\rangle$ is defined and contains $0$. Let $X(a_1,a_2,\ldots,a_n)$ denote the classifying space for the embedding problem
$\mathbf E(\kappa(a_1),\kappa(a_2),\ldots,\kappa(a_n))$.
\end{defn}
One of the main results of \cite{HW} (Theorem 1.1) is the construction of a splitting variety for the order-$3$ Massey product $\langle \kappa(a_1),\kappa(a_2),\kappa(a_3)\rangle$. However the existence of such a variety is an easy consequence of our results:
\begin{thm}\label{11.5} The construction $X(a_1,a_2,\ldots,a_n)$ is a splitting variety for the order-$n$ Massey product $\langle \kappa(a_1),\kappa(a_2),\ldots,\kappa(a_n)\rangle$.
\end{thm}
\begin{proof} This is an immediate consequence of Theorems \ref{homo_space} and \ref{11.3}.
\end{proof}

\section{A transcendental counter-example to weak approximation}

By slight abuse of notation for every $y\in |F|$ let
$$j_y:H^2(\Gamma(\mathbf{E}),\overline F^*)\longrightarrow 
H^2(\Gamma_y(\mathbf{E}_y),\overline F_y^*)$$
denote the composition $(\eta_y)_*\circ(\textrm{id}_{G_1}\times\iota_y)^*$.
\begin{lemma}\label{l:WA_1} Let $\mathbf{E}$ be an embedding problem over a global field $F$ such that $\textrm{\rm Sol}_{\mathbb A}(\mathbf{E})\neq\emptyset$.  If there exists a place $y\in |F|$  and an element $c\in  H^2(\Gamma(\mathbf{E}),\overline F^*)$ such that the map
$$\textrm{\rm Sol}(\mathbf{E}_y)\longrightarrow H^2(\Gamma_y,\overline F_y^*)$$
given by the rule:
$$h \mapsto s(h)^*(j_y(c))$$
is not constant, then
$$\textrm{\rm Sol}_{\mathbb A}^{\textrm{\rm Br}}(\mathbf{E})\subsetneq
\textrm{\rm Sol}_{\mathbb A}(\mathbf{E}).$$
\end{lemma}
\begin{proof} Let $h=\prod_{x\in|F|}h_x\in\textrm{\rm Sol}_{\mathbb A}(\mathbf{E})$ be an ad\`elic solution. If $(h,c)\neq 0$ then $h\not\in\textrm{\rm Sol}_{\mathbb A}^{\textrm{\rm Br}}(\mathbf{E})$ by definition. Otherwise choose an $h'_y\in \textrm{\rm Sol}(\mathbf{E}_y)$ such that
$$s(h'_y)^*(j_y(c)) \neq s(h_y)^*(j_y(c)).$$
Now for every $x\in |F|$ such that $x\neq y$ we set $h'_x=h_x$. Then $h'=\prod_{x\in|F|}h'_x\in\textrm{\rm Sol}_{\mathbb A}(\mathbf{E})$ but
$$(h',c)=(h,c)+s(h'_y)^*(j_y(c))-s(h_y)^*(j_y(c))\neq 0.$$
\end{proof}
\begin{defn} Let $\mathbf{E}$ be an embedding problem over a global field $F$ described by the diagram (\ref{1.0.1}). We shall say that a place $x\in |F|$ \emph{splits in $\mathbf{E}$} if the restriction $\psi\circ\iota_x:\Gamma_x\to G_2$ is trivial.
\end{defn}
\begin{lemma} Let $\mathbf{E}$ be an embedding problem over a global field $F$ such that  $\textrm{\rm Ker}(\mathbf{E})_{ab}=0$ and $\textrm{\rm Sol}_{\mathbb A}(\mathbf{E})\neq\emptyset$. Let $y\in |F|$ be a place that splits in $\mathbf{E}$, and assume that the pairing:
\begin{equation}\label{10.3.1}
\textrm{\rm Hom}(\Gamma_y,\textrm{\rm Ker}(\mathbf{E}))\times H^2(\textrm{\rm Ker}(\mathbf{E}),\!\!\!\!\!\!\bigoplus_{p\neq\textrm{\rm char}(F)}\!\!\!\!\!\!\mathbb Q_p/\mathbb Z_p)^{\Gamma} \to H^2(\Gamma_y,\!\!\!\!\!\!\bigoplus_{p\neq\textrm{\rm char}(F)}\!\!\!\!\!\!\mathbb Q_p/\mathbb Z_p)
\end{equation}
given by the rule $(a,c)\mapsto a^*(c)$ is non-zero. Then we have:
$$\textrm{\rm Sol}_{\mathbb A}^{\Br}(\mathbf{E})\subsetneq\textrm{\rm Sol}_{\mathbb A}^{\textrm{\rm Br}_1}(\mathbf{E})=
\textrm{\rm Sol}_{\mathbb A}(\mathbf{E}).$$
\end{lemma}
\begin{proof} For simplicity let
$$\mu_{\infty}=\bigoplus_{p\neq\textrm{\rm char}(F)}\!\!\!\!\!\!\mathbb Q_p/\mathbb Z_p,$$
equipped with the trivial $\textrm{Ker}(\mathbf{E})$-action. Since $\textrm{Ker}(\mathbf{E})_{ab} = 0$ we have $$\textrm{Sol}_{\mathbb A}^{\Br_1}(\mathbf{E})=\textrm{Sol}_{\mathbb A}(\mathbf{E})$$
by Theorem \ref{4.10}. Therefore it will be enough to show that 
$$\textrm{Sol}_{\mathbb A}^{\Br}(\mathbf{E})\subsetneq
\textrm{Sol}_{\mathbb A}(\mathbf{E}).$$
Choose a pair $(a,c) \in \textrm{\rm Hom}(\Gamma_y,\textrm{\rm Ker}(\mathbf{E}))\times H^2(\textrm{\rm Ker}(\mathbf{E}),
\mu_{\infty})^{\Gamma} $ such that $a^*(c)\neq0$. Since $\textrm{\rm Ker}(\mathbf{E})_{ab}=0$ the map:
$$p:H^2(\Gamma(\mathbf{E}),\overline F^*) \lrar H^2(\Ker(\mathbf E),\mu_{\infty})^{\Gamma}$$
in the short exact sequence in Lemma \ref{2.2} is surjective. Let $\tilde{c}\in H^2(\Gamma(\mathbf{E}),\overline F^*)$ be an element such $p(\tilde{c}) = c$. Since $y$ splits we have  $\Gamma_y(\mathbf{E})\cong
\Ker(\mathbf{E})\times\Gamma_y$ and
Sol$(\mathbf{E}_y)$ is just Hom$(\Gamma_y,\textrm{\rm Ker}(\mathbf{E}))$. We have a commutative diagram:
$$\xymatrix{
\textrm{Sol}(\mathbf{E}_{y})\times H^2(\Gamma(\mathbf{E}),\overline F^*)
 \ar@<7ex>[d]^{p} \ar[r] & H^2(\Gamma_y,\overline F^*) \ar@{=}[d]  \\
\textrm{Hom}(\Gamma_y,\textrm{Ker}(\mathbf{E})) \ar@<7ex>[u]^{\cong}_{\sigma} \times H^2(\textrm{Ker}(\mathbf{E}),
\mu_{\infty})^{\Gamma}\ar[r] & H^2(\Gamma_y,\mu_{\infty}),}$$
where the upper horizontal map is given by the rule:
$$(h,d)\mapsto s(h)^*(j_y(d))$$
and the lower horizontal map is the pull-back with respect to homomorphisms. The commutativity of the diagram above follows from the fact that for every $b \in\textrm{Hom}(\Gamma_y,\textrm{\rm Ker}(\mathbf{E}))$
the diagram:
$$\xymatrix{
\Gamma_y \ar[rd]_{\textrm{Id}_{\Gamma_y}\times b}\ar[r]^{\sigma(b)} & \Gamma_y(\mathbf{E}) \ar[r] \ar[d]^{\cong} & \Gamma(\mathbf{E}) \\ 
\empty &          \Gamma_y \times \Ker(\mathbf{E}) \ar[r]^-{\pi_2}& \Ker(\mathbf{E}) \ar[u]\\}$$ 
is commutative. Now let $0\in\textrm{Hom}(\Gamma_y,\textrm{\rm Ker}(\mathbf{E}))$ be the trivial map. By the above we have: 
$$s(\sigma(0))^*(j_y(\tilde{c}))=0^*(c)=0, \; s(\sigma(a))^*(j_y(\tilde{c}))=a^*(c)\neq 0,$$
so by Lemma ~\ref{l:WA_1} we have:
$$\textrm{\rm Sol}_{\mathbb A}^{\textrm{\rm Br}}(\mathbf{E})\subsetneq
\textrm{\rm Sol}_{\mathbb A}(\mathbf{E}).$$
\end{proof}
\begin{prop}\label{10.4} Let $\mathbf{E}$ be an embedding problem over a number field $F$ such that $\textrm{\rm Ker}(\mathbf{E})_{ab} = 0$ and $\textrm{\rm Sol}_{\mathbb A}(\mathbf{E}) \neq \emptyset$. Let $y\in |F|$ be a real place such that $y$ splits in $\mathbf{E}$, and assume further that the pairing
\begin{equation}\label{10.4.1}
\textrm{\rm Hom}(\mathbb Z/2\mathbb Z,\textrm{\rm Ker}(\mathbf{E}))\times
H^2(\textrm{\rm Ker}(\mathbf{E}),\mathbb Z/2\mathbb Z)
\to H^2(\mathbb Z/2\mathbb Z,\mathbb Z/2\mathbb Z)
\end{equation}
given by the rule $(a,c)\mapsto a^*(c)$ is non-zero. Then we have:
$$\textrm{\rm Sol}_{\mathbb A}^{\Br}(\mathbf{E})\subsetneq\textrm{\rm Sol}_{\mathbb A}^{\textrm{\rm Br}_1}(\mathbf{E})=
\textrm{\rm Sol}_{\mathbb A}(\mathbf{E}).$$
\end{prop}
\begin{proof} Because $y$ is real we have $\Gamma_y\cong\mathbb Z/2\mathbb Z$. Since $y$ splits in $\mathbf{E}$ the group $\Gamma_y$ acts trivially on $\textrm{\rm Ker}(\mathbf{E})$, and so on the group $H^2(\textrm{\rm Ker}(\mathbf{E}),\mathbb Z/2\mathbb Z)$, and hence the pairing in (\ref{10.4.1}) is the same as the pairing in (\ref{10.3.1}). So the claim follows from the previous lemma.
\end{proof}
\begin{example}\label{the_example} Let 
\begin{equation}\label{10.5.1}
1 \lrar A_5 \x{i}{\lrar} S_5 \x{p}{\lrar} \mathbb Z/2\mathbb Z \lrar 1
\end{equation}
be the short exact sequence where $i$ the inclusion of $A_5$ into $S_5$. Consider now the embedding problem over $\mathbb Q$ defined by the diagram:
\begin{equation}\label{10.5.2}
\CD @.\Gamma\\
@.@V\psi VV\\
S_5@>p>>\mathbb Z/2\mathbb Z\endCD
\end{equation}
where the map $\psi$ corresponds to the extension $\mathbb Q(\sqrt(d))  /\mathbb Q$ for some square free positive integer $d$. Since (\ref{10.5.1})
splits we have Sol$(\mathbf{E})\neq 0$. Moreover the unique real place of $\mathbb Q$ splits in $\mathbf E$, since $d$ is positive. It is well known that $A_5^{ab}$ is trivial and $H^2(A_5,\mathbb Z) = \mathbb Z/2\mathbb Z$. In fact the nontrivial element in $H^2(A_5,\mathbb Z/2\mathbb Z)$ corresponds
to the central extension:
$$1\longrightarrow \mathbb Z/2\mathbb Z \longrightarrow I \longrightarrow A_5 \longrightarrow 1,$$
where $I$ is the binary icosahedral group. Now when pulling back by the map $\mathbb Z/2\mathbb Z \to A_5$ sending $1$
to $(1,2)(3,4)$ we get the sequence
$$1\longrightarrow \mathbb Z/2\mathbb Z \longrightarrow \mathbb Z/4\mathbb Z \longrightarrow \mathbb Z/2\mathbb Z
\longrightarrow 1$$
which corresponds to the non-trivial element in $H^2(\mathbb Z/2\mathbb Z,\mathbb Z/2\mathbb Z)$. So Proposition \ref{10.4} implies that
$$\textrm{\rm Sol}_{\mathbb A}^{\Br}(\mathbf{E})\subsetneq\textrm{\rm Sol}_{\mathbb A}^{\textrm{\rm Br}_1}(\mathbf{E})=
\textrm{\rm Sol}_{\mathbb A}(\mathbf{E}).$$
\end{example}

\end{document}